\documentclass{amsart}

\usepackage[utf8]{inputenc}

\usepackage{amsthm}
\usepackage{amsmath}
\usepackage{amsfonts}
\usepackage{ amssymb }
\usepackage{todonotes}

\usepackage{hyperref}
\usepackage{graphicx}
\usepackage{graphicx}
\usepackage{caption}
\usepackage{subcaption}

\newtheorem{theorem}{Theorem}[section]
\newtheorem*{theorem*}{Theorem}
{\theoremstyle{definition}
\newtheorem{defn}[theorem]{Definition}}
\newtheorem{lemma}[theorem]{Lemma}

\newtheorem{corollary}[theorem]{Corollary}
\newtheorem{remark}[theorem]{Remark}
{\theoremstyle{remark}
}

\newcommand{\T}{\mathcal{T}}

\newcommand{\h}{\mathbb{H}}
\newcommand{\R}{\mathbb{R}}
\newcommand{\C}{\mathbb{C}}
\newcommand{\Q}{\mathbb{Q}}

\makeatletter
\newtheorem*{rep@theorem}{\rep@title}           
\newcommand{\newreptheorem}[2]{%
\newenvironment{rep#1}[1]{%
 \def\rep@title{#2 \ref*{##1}}%
 \begin{rep@theorem}}%
 {\end{rep@theorem}}}
\makeatother

\newreptheorem{theorem}{Theorem}
\newcommand\restr[2]{{% we make the whole thing an ordinary symbol							%Restriction symbol, use as \restr{f}{A}
  \left.\kern-\nulldelimiterspace % automatically resize the bar with \right
  #1 % the function
  \vphantom{\big|} % pretend it's a little taller at normal size
  \right|_{#2} % this is the delimiter
  }}

\title{The Homeomorphism Problem for Hyperbolic Manifolds I.}
\author{Joe Scull}
\date{}

\begin{document}

\begin{abstract}
We give a bounded runtime solution to the homeomorphism problem for closed hyperbolic 3-manifolds. This is an algorithm which, given two triangulations of hyperbolic 3-manifolds by at most $t$ tetrahedra, decides if they represent the same hyperbolic 3-manifold with runtime bounded by
\[ 2^{2^{t^{O(t)}}}.\]
We do this by first finding a hyperbolic structure on each manifold given as a geometric triangulation  and then comparing the two as geometric manifolds.

\end{abstract}
\maketitle

\tableofcontents
\newpage

\section{Introduction}

A natural question in the study of manifolds is whether two manifolds from a given class $\mathcal{C}$ are homeomorphic. More precisely, we ask whether an algorithm exists which, given two triangulations and the promise that the triangulations are manifolds in $\mathcal{C}$, can decide whether or not they are homeomorphic. This is called the homeomorphism problem for $\mathcal{C}$.

In this paper we shall focus on classes of orientable manifolds and from now all manifolds are assumed to be orientable without comment.
In dimensions 1 and 2, the homeomorphism problem for closed connected manifolds is solved quite simply, as there is only one such $1$-manifold and closed connected $2$-manifolds are classified by their Euler characteristic and their orientability.
In dimensions 4 or above, a result of Markov \cite{Markov58} states that the homeomorphism problem, even for closed connected manifolds, is unsolvable.

In dimension 3 the problem hits a sweet spot of being tractable while still being quite difficult. 
For example all known solutions to the homeomorphism problem for $3$-manifolds use Thurston's Geometrisation Conjecture \cite{Thurston82}, and thus require Perelman's resolution of the conjecture \cite{Perelman02}\cite{Perelman03a}\cite{Perelman03b}.

In \cite{Sela95} a solution to the homeomorphism problem for closed connected oriented 3-manifolds is given as a corollary of Thurston's Geometrisation before Pereleman had proved it. A more accessible description of this algorithm is given in the book of Bessieres et al. \cite{Bessieres2010}, citing both a series of lectures by Jaco and Sela's work. In these algorithms, hyperbolic manifolds are dealt with using Sela's solution to the isomorphism problem for torsion free hyperbolic groups.
Scott and Short also provide an alternative algorithm \cite{ScottShort14} which avoids Sela's work.

With the decidability of the homeomorphism problem settled, the next question is naturally whether we can give bounds on runtime. Partial results for certain subclasses of 3-manifolds had previously been given by Mijatovic \cite{Mijatovic03}\cite{Mijatovic04} \cite{Mijatovic05}. 
Currently, the best known result is given by Kuperberg \cite{Kuperberg19} in which he proves that the homeomorphism problem for closed oriented 3-manifolds admits a bounded runtime solution, though the bound is partially unknown. 
That is, Kuperberg's solution has runtime bounded by a tower of exponentials in the number of tetrahedra, but the height of this tower is unknown. 
The question asked at the end of Kuperberg's paper is what height might such a tower have, whether for his or an improved algorithm? 

A common thread in all the above approaches is that they deal with hyperbolic 3-manifolds using a bespoke algorithm. In Kuperberg's case for example, he algorithmically finds the geometric decomposition of the 3-manifold and then compares the (Seifert fibred or hyperbolic) pieces.
Thus a solution to the homeomorphism problem for all closed 3-manifolds seems destined to require one for hyperbolic 3-manifolds.

In the following paper we present a solution to the homeomorphism problem for closed hyperbolic 3-manifolds with runtime bounded by a tower of exponentials of height 4.

\begin{theorem}\label{MainTheorem}
If $M_1$ and $M_2$ are triangulated closed hyperbolic 3-manifolds, each triangulated by less than $t$ tetrahedra, then we can decide in time bounded by 
\[  2^{2^{t^{O(t)}}}\]
whether or not they are homeomorphic. 
\end{theorem}

Note here that the algorithm which decides whether the manifolds are homeomorphic takes as input solely the combinatorial data of the triangulations. In particular it requires no information regarding the hyperbolic structure. This means that the algorithm needs to be capable of finding hyperbolic structures as well as comparing them.  

In some parts the methods of this paper are a careful application of the ideas of Kuperberg's paper \cite{Kuperberg19}, paying special attention to impact on computational complexity. In some important steps, different methods to those envisaged by Kuperberg were required.

As part of the proof of Theorem 1.1 we also prove the following result, adapted from Theorem \ref{TheoremListGeometricTriangulations}, which give a bounded runtime algorithm for finding the hyperbolic structure of a given hyperbolic 3-manifold $M$, in the form of a geometric triangulation of $M$. Finding the hyperbolic structure of a hyperbolic 3-manifold was already proven to be algorithmically possible by Manning \cite{Manning01} (Scott and Short's Paper also touches on this \cite{ScottShort14}) and of course Kuperberg's paper \cite{Kuperberg19} provides a bounded runtime algorithm. The following provides the best known bound on runtime for such an algorithm.

\begin{theorem}\label{Thm1.2}
Given a triangulation of a hyperbolic 3-manifold $M$ by $t$ tetrahedra, there is an algorithm which produces a geometric triangulation of $M$ consisting of less than

\[ 2^{t^{O(t)}}\]
tetrahedra in time bounded by
\[ 2^{2^{t^{O(t)}}}.\]
Furthermore, this geometric triangulation has the property that all edges have length less than the injectivity radius of the manifold.
\end{theorem}

Theorem \ref{Thm1.2} requires the following two results which may be of independent interest. 

\begin{theorem} \label{Thm1.3}
Given a triangulation of a closed hyperbolic 3-manifold $M$ by $t$ tetrahedra, $M$ admits a geometric triangulation where the number of tetrahedra in the geometric triangulation is bounded by

\[  2^{t^{O(t)}}.\]

Furthermore there is some subdivision of the triangulation of $M$ and a surjective degree one simplicial map from this subdivision to our geometric triangulation. The size of the subdivision is bounded similarly.
\end{theorem}

Theorem \ref{Thm1.3} is adapted from Corollary \ref{GeometricSubdivisionBound} which proves a more general result for closed hyperbolic $n$-manifolds, and in it we show that the geometric triangulation is in fact a simplicial quotient of a subdivision of $M$. We shall define simplicial quotients later, but in particular the simplicial map mentioned in the above is non-degenerate on all simplices. 

Finally, to find a simplicial quotient of a certain subdivision, we first need to be able to find this subdivision in bounded runtime. Such an algorithm exists in the works of Kuperberg \cite{Kuperberg19} and Kalelkar and Phanse \cite{KalelkarPhanse19}. Our algorithm differs in that we give an explicit sequence of a bounded number of Pachner moves to get from a triangulation to its subdivision. 

The following is adapted from Theorem \ref{SubdivisionMoves} and Remark \ref{RemarkShellabilityTechnicality}.

\begin{theorem}\label{Thm1.4}
Suppose we have a 3-manifold with a triangulation $\T_1$ by $t$ tetrahedra and a subdivision $\T_2$ by $T$ tetrahedra. Then $\T_1$ and $\T_2$ are related by a sequence of Pachner moves of length bounded by

\[ O(tT).\] 
\end{theorem}

In fact, we define composite moves (combinations of Pachner moves) which preserve subdivisions, and so we can list all subdivisions of a certain complexity by performing all sequences of these composite moves with the length of the sequence bounded as above.

\subsection{Future Work}
The obvious next step to an algorithm for the general case of closed 3-manifolds would be the case of finite volume hyperbolic 3-manifolds, as these arise naturally in Thurston's geometric decomposition.
In fact, the generality of the methods used in this work suggest that the case of $n$-dimensional hyperbolic manifolds (either closed or finite volume) may be tractable. 
Again, if we don't require bounded runtime, the homeomorphism problem is already known to be decidable for this class of manifolds, via an application of a solution to the isomorphism problem for their fundamental groups \cite{DahmaniGroves08}. 
In upcoming work of the author \cite{Scull}, we aim to prove the existence of a bounded runtime algorithm in the $n$-dimensional finite volume case. Because of this, we sometimes prove or cite results in more generality than is necessary, to avoid reproving theorems in future work. That said, sticking to the closed 3-dimensional case allows for both a tighter bound, and a tighter story. Though the upcoming work will use many of the same ideas, removing the closed and 3-dimensional hypotheses introduces many technicalities, so the hope is that the reader can view this paper as a blueprint for how to solve the finite volume problem in all dimensions, and if they are interested in the granular details, they can read the upcoming paper \cite{Scull}.

\subsection{Outline of the Algorithm}

Our algorithm works as follows. Let $M$ be a $3$-manifold triangulated by $t$ tetrahedra and let $T > t$ be a natural number. We can first list all subdivisions with number of tetrahedra bounded by $T$, then list all their degree one simplicial quotients, then check whether these simplicial quotients admit the structure of a geometric triangulation. Thus we have a method for listing geometrically triangulated hyperbolic 3-manifolds which are degree one simplicial quotients of subdivisions of $M$. We then show that there is some choice of $T$ which guarantees that some entry on the list is in fact a geometric triangulation of $M$ itself and that it admits a geometric structure with strict bounds on the lengths of edges of tetrahedra relative to injectivity radius. 
Having generated a list for this choice of $T$, we then attempt to produce the sublist which consists entirely of geometrically triangulated hyperbolic manifolds homeomorphic to $M$. This is quite tricky,  we accomplish this by searching for surjections amongst the fundamental groups of the manifolds in our list. Mostow Rigidity plus the fact that all the groups are Hopfian can then be used to show that this property of $\pi_1$-surjecting onto all other elements of the list is a necessary and sufficient condition for  a geometric triangulation in our list to be homeomorphic to $M$. 

Having found our geometric triangulations we can then appeal to an existing result of Kalelkar and Phanse \cite{KalelkarPhanse19} to compare them. This result uses both the bounds on the number of tetrahedra in our triangulations, as well as the bounds on edge length.

\section{Outline of the Paper}
 
In Section \ref{SecPolynomial} we introduce a useful algorithm due to Grigoriev and Vorobjov \cite{GrigorievVorobjov88} for finding solutions to a system of polynomial inequalities with solution size and runtime bounded by a function of the size of the system. 

In Section \ref{SecPolyhedron} we introduce Poincar\'e's Polyhedron Theorem. Given a triangulation $\T$ of a $3$-manifold $M$, we show how it can be used to generate a system of equations which have solution iff $\T$ admits the structure of a geometric triangulation (each tetrahedron is a geodesically embedded hyperbolic tetrahedron).

Section \ref{SecGeometricTriangulations} is then a careful application of the previous two sections, giving a bound on the runtime of an algorithm which checks whether a triangulation admits a geometric structure. While we're there we also show how given a triangulation one could use this algorithm find all degree one simplicial quotients of a given triangulation which admit this structure of a geometric triangulation.

Our next step in Section \ref{SecPachner} is about using Pachner moves to list all subdivisions of a triangulation of a given complexity. We do this by restricting which combinations of moves we can use, and in doing so creating sequences of combinatorial moves which preserve the property of being a subdivision.

As noted in our outline of the algorithm, we now need to know what complexity of subdivision guarantees that one such subdivision will have a geometric triangulation of $M$ as a degree one simplicial quotient. We do this by using a result from a previous paper of the author \cite{Scull21} which provides a homotopy equivalence of the original triangulation onto our hyperbolic manifold such that the images of the tetrahedra form a cover by geodesically embedded hyperbolic tetrahedra. We can then subdivide this cover to a true geometric triangulation and subdivide its preimage to a subdivision of $M$ so that the homotopy equivalence is simplicial.
We also show in this section that the geometric triangulation we found has its edge lengths bounded relative to the injectivity radius of the manifold and that we have a bound on its number of tetrahedra.

In Section \ref{SecFindingGeometricTriangulations} we recap the results of the previous sections which allow us to construct a list of geometrically triangulated degree one simplicial quotients of subdivisions of $M$. We then show how to refine this list so it only consists of geometric triangulations of $M$, and in fact then find one which has edge lengths bounded by injectivity radius.

Finally, in Section \ref{SecComparingGeometricTriangulations} we use a result of Kalelkar and Phanse \cite{KalelkarPhanse19} to compare the respective geometric triangulations. Their algorithm has runtime a function of the lengths of the edges and the injectivity radius of the manifold. Here our knowledge of the edge length bounds makes for an algorithm which runs in time polynomial in the number of tetrahedra in the geometric triangulation.  The original triangulations represent the same manifold iff their geometric triangulations do, and so the algorithm is complete.

\section{Systems of Polynomial Inequalities} \label{SecPolynomial}

The following is a survey of the results in \cite{GrigorievVorobjov88} adapted from a similar survey in \cite{Grigoriev86}. Furthermore, this survey also appears in previous work of the author \cite{Scull21}, but has been included here for the purposes of self-containedness. We show here how to find solutions to systems of polynomials which are bounded in terms of the size of the system. \\ \medskip
Let a system of polynomial inequalities
\begin{equation*}
f_1 > 0, \ldots f_m > 0, f_{m+1} \geq 0, \ldots , f_{\kappa} \geq 0
\end{equation*}
be given, where the polynomials $f_i \in \Q[X_1, \ldots, X_N]$ satisfy the following bounds (notation defined below)
\begin{equation*}
    deg_{X_1,\ldots,X_N}(f_i) < d, \qquad l(f_i) < M, \qquad 1 \leq i \leq \kappa.
\end{equation*}
We define the degree of a monomial to be the sum of the exponents of the monomial and then $deg_{X_1,\ldots,X_N}(f_i)$ is the maximum degree of a monomial in $f_i$. The length, or complexity $l$ is defined on rational numbers by $l (\frac{p}{q}) = log_2(\lvert pq \rvert +2)$ and on polynomials it is defined as the maximum complexity among its coefficients.

Let $\alpha = (\alpha_1, \ldots, \alpha_N)$ be an $N$-tuple of algebraic numbers. We can represent each $\alpha_i$ in terms of a primitive element, $\theta$, of the field $\Q (\alpha_1, \ldots, \alpha_N) = \Q(\theta)$. We provide an irreducible polynomial $\Phi (X) \in \Q (X)$ of which $\theta$ is a root and an interval $(\beta_1, \beta_2) \subseteq \R$ with endpoints in $\Q$ which determines $\theta$ among the roots of $\Phi$. With $\theta$ defined, one has $\alpha_i = \sum_j \alpha_i^{(j)} \theta^j$ for $\alpha_i^{(j)} \in \Q$. We use this notation below to describe bounds on the information defining certain algebraic solutions of a system of polynomials.

\begin{theorem}[Grigoriev \cite{Grigoriev86}] \label{Grigoriev}

Suppose we are given a system of polynomial inequalities defined as above. For each connected component of the solution set, there exists a solution $\alpha = (\alpha_1, \ldots, \alpha_N) $ where each $\alpha_i$ is an algebraic number and for which the following are true (see above definitions):
\begin{center}

$deg(\Phi) \leq (\kappa d )^{O(N)}; \qquad l(\Phi), l(\alpha_j^{(i)}), l(\beta_1), l(\beta_2) \leq M (\kappa d)^{O(N)}$
\end{center}

Furthermore, an algorithm exists which finds such a solution for each connected component in time polynomial in $M (\kappa d)^{N^2}$.
\end{theorem}

The following corollary follows the same method of proof as that of Lemma 8.9 in \cite{Kuperberg19}, but for self-containedness, we provide it here.
\begin{corollary}
$\frac{1}{2^{M (\kappa d)^{O(N)}}} \leq \lvert \theta \rvert \leq 2^{M (\kappa d)^{O(N)}}$
\end{corollary}

\begin{proof}
Note that we can scale $\Phi$ so that it is an integer polynomial without negating the statement that $l(\Phi)\leq M (\kappa d)^{O(N)}$. So we may assume that the coefficients, $\gamma_i$ of $\Phi$ are integers, and hence the $\gamma_i$ have size bounded by $2^{l(\Phi)} -2$.
Now note that as $\Phi (\theta) = 0$ we get that 
\[ -\gamma_{deg(\Phi)} \theta^{deg(\Phi)} = \gamma_0 + \ldots +\gamma_{deg(\Phi)-1}\theta^{deg(\theta)-1} \]
and so either $\lvert \theta \rvert \leq 1$ in which case the upper bound follows or
\[ \lvert \theta \rvert = \left\lvert \frac{\gamma_0 + \ldots +\gamma_{deg(\Phi)-1}\theta^{deg(\Phi)-1} }{\theta^{deg(\theta)-1}\gamma_{deg(\Phi)}} \right\rvert  \leq \sum_{i=0}^{deg(\Phi)-1} \lvert \gamma_i \rvert\ \] 
Similarly, as $\theta^{-1}$ is a root of 
\[ \Psi(x) = x^{deg(\Phi)} \gamma_0 + \ldots + \gamma_{deg(\Phi)} = x^{deg(\Phi)} \Phi(x^{-1}) \]
either $ \lvert \theta^{-1} \rvert \leq 1$ or 
\[   \lvert \theta^{-1} \rvert \leq \sum_{i=0}^{deg(\Phi)-1} \lvert \gamma_i \rvert .\]
Applying the bounds on the $\gamma_i$ gives
\begin{equation*}
\lvert \theta \rvert \leq  \sum_{i=0}^{deg(\Phi)-1} \gamma_i \leq \sum_{i=0}^{deg(\Phi)-1} 2^{l(\Phi)} \leq deg(\Phi) 2^{l(\Phi)}
\end{equation*}
and similarly,
\begin{equation*}
\lvert \theta \rvert \geq  \frac{1}{deg(\Phi) 2^{l(\Phi)}}.
\end{equation*}
Finally, Theorem \ref{Grigoriev} tells us that 
\begin{equation*}
    deg(\Phi) 2^{l(\Phi)} \leq (\kappa d )^{O(N)} 2^{M (\kappa d)^{O(N)}} \leq 2^{M (\kappa d)^{O(N)}}
\end{equation*}
and the statement follows.
\end{proof}
From this, we say something about the size of the solutions themselves.

\begin{corollary} \label{variablebound}
Let $\alpha$ be as in Theorem \ref{Grigoriev}, then for each $i$, if $\alpha_i \neq 0$
\begin{equation*}
    \frac{1}{2^{M ((\kappa+2N) d)^{O(N)}}} \leq \lvert \alpha_i \rvert \leq 2^{M (\kappa d)^{O(N)}} .
\end{equation*}
\end{corollary}

\begin{proof}
First see that we can achieve an upper bound on $\lvert \alpha_i \rvert$ by applying the triangle inequality in the following manner:
\begin{equation*}
    \lvert \alpha_i \rvert 
    \leq \lvert \sum_{j=0}^{deg(\Phi)}\alpha_i^{(j)} \theta^j \rvert 
    \leq \sum_{j=0}^{deg(\Phi)}\lvert \alpha_i^{(j)} \rvert \lvert \theta^j \rvert 
    \leq \sum_{j=0}^{deg(\Phi)} 2^{M (\kappa d)^{O(N)}} 2^{M (\kappa d)^{O(N)}} 
\end{equation*}
\begin{equation*}
    \leq ((\kappa d)^{O(N)}+1)  2^{M (\kappa d)^{O(N)}} 2^{M (\kappa d)^{O(N)}}
    \leq  2^{M (\kappa d)^{O(N)}} .
\end{equation*}
Note that a lower bound on $\lvert \alpha_i\rvert$ is equivalent to an upper bound on $\lvert \alpha_i^{-1} \rvert$, so we can modify our system of polynomial inequalities by adding some $\beta_i$ such that $\alpha_i \beta_i = 1$ for all $i$. This is represented by two inequalities and one new variable for each of the original variables. Now applying the theorem to this new system of inequalities gives an upper bound for $\lvert \beta_i \rvert$ as above except $\kappa$ is replaced by $\kappa +2N$, as $O(2N)$ and $ O(N)$ are the same. Taking reciprocals of both sides gives us the lower bound in the statement.
\end{proof}

Before we move on we should note that the theorem requires that the system be given as polynomial inequalities. When we create our system, we shall allow for equalities as well, which can be replaced by a pair of inequalities. Because of big O notation, doubling the number of polynomials will not affect the final bound and so we shall ignore this distinction from here on.

\section{Poincar\'e's Polyhedron Theorem and Mostow Rigidity} \label{SecPolyhedron}

In this section we shall first state Poincar\'e's Polyhedron Theorem and then show how it gives us a means of checking whether an assignment of hyperbolic structure to the tetrahedra induces a hyperbolic structure on the gluing. 

Though we shall not state the theorem in full generality, it is possible to state this for different geometries, types of building block and with different resulting objects. For the full statement, as well as a brief history of the Poincar\'e Polyhedron theorem, the reader is directed to \cite{EpsteinPetronio94}, from which this entire section is adapted.

We start with a definition due to Thurston \cite{Thurston97}, who calls this concept a gluing, which allows us to generalise the concept of the triangulation of a manifold by a simplical complex.

\begin{defn}[Gluing]
An n-dimensional (rectilinear) gluing consists of a finite set of $n$-simplices, a choice of pairs of codimension one faces such that each face appears in exactly one of the pairs and an affine identification map between the faces of each pair.
We also call the quotient space derived from this process a gluing.
\end{defn}

If we take each of the $n$-simplices to be the standard $n$-simplex then affine identification maps are uniquely determined on the faces by where the vertices are sent. Thus this can be made a purely combinatorial definition.
Note that taking the double barycentric subdivision of a gluing gives a simplicial complex.

\begin{lemma} [Proposition 3.2.7 \cite{Thurston97} ] 
A three-dimensional gluing is a manifold iff in the double barycentric subdivision the link of every vertex is homeomorphic to a 2-sphere.
\end{lemma}

Note that checking whether the link of every vertex in a three-dimensional gluing $M$ is homeomorphic to a 2-sphere is a procedure which has running time polynomial in the number of tetrahedra in $M$ \cite{FomenkoMatveev97}. So provided with a three-dimensional gluing, one can check whether it is indeed a manifold in polynomial time.

\begin{defn}
A triangulation of an $n$-manifold $M$ is an $n$-dimensional gluing paired with a homeomorphism to $M$.
\end{defn}

\begin{defn}[Simplicial Map]
A continuous map $f: M \rightarrow N$ between $n$-dimensional gluings is simplicial if it induces a simplicial map on its barycentric subdivision.
\end{defn}

\begin{defn}[Simplicial Quotient]
Given an $n$-dimensional gluing $M$, and a collection of simplicial identifications $ \Delta \rightarrow \Delta'$ for  $\Delta, \Delta'$ $n$-simplices in $M$. Consider the quotient space $Q$ of all these identifications. If $Q$ is itself an $n$-dimensional gluing then we call it a simplicial quotient of $M$.
\end{defn}

%Checking whether a quotient such as that defined in the above definition is indeed a gluing is equivalent to checking that no face in the quotient is identified to more than one other face, and this can be done in time polynomial in the number of $n$-simplices. 

\begin{figure}[ht]
\centering
    \includegraphics[width = 0.8\linewidth]{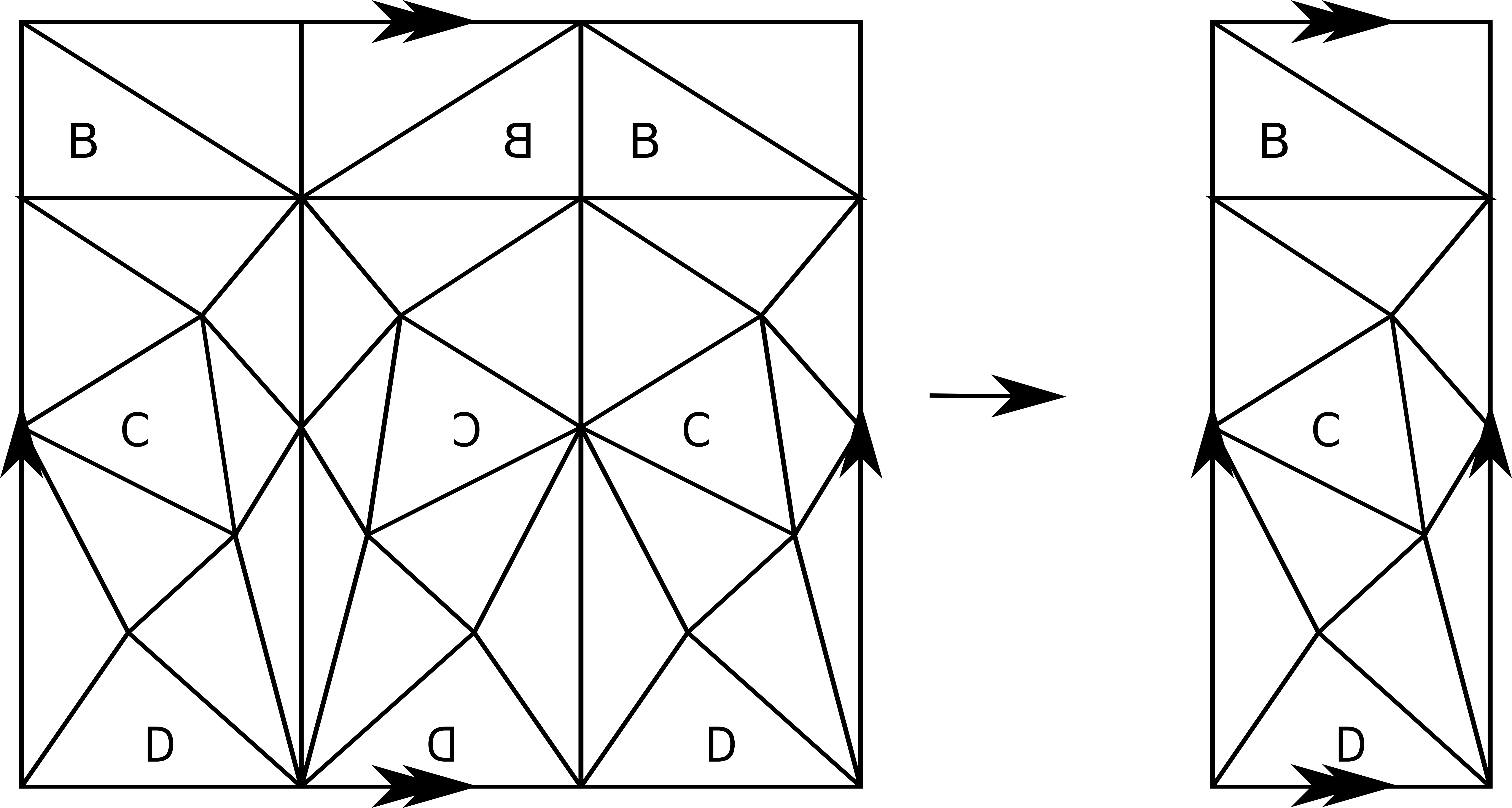}
    \caption{An example of a degree one simplicial quotient map from the torus to itself. The faces labelled with letters indicate how faces are identified to make the letters line up.}
\end{figure}

\begin{defn}[Model Tetrahedra]
Let $M$ be a triangulated $3$-manifold  consisting of $t$ tetrahedra. Let $\mathcal{M}$ to be a set of $t$ geodesic tetrahedra in the upper half space model of $\h^3$ such that each is identified with a distinct simplex in $M$ by a simplicial map. We call such a set a set of model tetrahedra for the triangulation of $M$.
\end{defn}

\begin{defn}[Dual Graph of a Triangulation]
We define the dual graph $\Lambda(M)$ of a triangulation to be the graph whose vertices are the $n$-simplices of the triangulation, and whose edges are codimension one faces along which two simplices meet.
\end{defn}

\begin{defn}
Consider some spanning tree $\Lambda'$ of $\Lambda(M)$ and  suppose $\mathcal{M}$ is such that the face pairings in $M$ can be realised by isometries. Then define a polyhedron $Y$ by gluing the tetrahedra in $\mathcal{M}$ only along faces corresponding to the edges of $\Lambda'$. If we also pick some base tetrahedron $\Delta_0$ in $\Lambda(M)$ then we can define a map $D: Y \rightarrow \h^3$ by first defining $D$ to be the identity on the tetrahedron $\Delta_0$ and then using $\Lambda'$ to define inductively on the other tetrahedra. That is to say that an edge of $\Lambda'$ connecting $\Delta_0$ to some $\Delta_1$ defines a unique embedding of $\Delta_1$ into $\h^3$ by choosing the embedding which agrees with that of $\Delta_0$ on the shared face and which ensure the interiors of their embeddings don't overlap, repeating this over all edges defines the map $D$.

\end{defn}

\begin{theorem}[Poincar\'e Polyhedron Theorem] \label{PoincarePolyhedronTheorem1}
 Let $M$, $\mathcal{M}, \Lambda(M)$ be as above and suppose the following hold:
\begin{itemize}
    \item (Edge Equations) For any two tetrahedra in $M$ which intersect in an edge, the corresponding edges in $\mathcal{M}$ are of equal length.
    \item (Angle Equations) For any edge of the triangulation, let $\alpha_1, \ldots \alpha_m$ be the dihedral angles at the corresponding edge in each of the tetrahedra of $\mathcal{M}$. Then $\sum \alpha_i = 2\pi$.
    \item
    (Orientation) There is an orientation of $\h^3$ with respect to which all the identifications in $\mathcal{M}$ are orientation preserving.
\end{itemize}

Then $M$ admits a hyperbolic structure such that when we restrict to any specific tetrahedron the structure is isometric to that given in $\mathcal{M}$. Furthermore, $M$ can be described as the quotient of the action of some uniform lattice $\Gamma \leq Isom^+(\h^3)$ where $\Gamma$ has the following properties. 

\begin{itemize}

\item If we pick some spanning tree $\Lambda'$ of $\Lambda(M)$ and define $Y$, $D$ as in the above definition, then $D$ is injective on the interior of $Y$ and has image a fundamental domain for the lattice $\Gamma$. 

\item For each edge $e$ of $\Lambda(M)$ not in $\Lambda'$ there are two faces $F_1, F_2$ in $Y$ corresponding to this edge, and there is a unique orientation preserving isometry $g_e$ (possibly $g_e = id.$) of $\h^3$ which takes $D(F_1)$ to $D(F_2)$.  The set of all the $g_e$ forms a generating set for the lattice $\Gamma$.

\end{itemize}

\end{theorem}

\begin{defn}
If a triangulation admits a set of model tetrahedra satsifying the premises of the above theorem, we say this triangulation admits the structure of a geometric triangulation, or simply admits a geometric/hyperbolic structure.
\end{defn}

\begin{remark} \label{LatticeGenTechnicality}
Suppose we satisfy the premises of the above theorem.
If we consider the tetrahedra in $\mathcal{M}$ all embedded in $\h^3$, then we can define a unique orientation preserving isometry $h_E$ of $\h^3$ for each oriented edge $E$ of $\Lambda(M)$ (also thought of as an ordered pair of faces $(F \subseteq \Delta,F' \subseteq \Delta')$) by requiring that it takes $F$ to $F'$. Note that as our tetrahedra satisfy the orientation requirement, this isometry will have the property that $h_E(\Delta)$ and $\Delta'$ have disjoint interiors.

The face pairings $g_e$ of the polyhedron $D(Y)$ defined above can then be described as a product of the face pairings $h_{E_i}$ where the $E_i$ make up a loop in $\Lambda(M)$ consisting of $e$ and the unique path in $\Lambda'$ between its endpoints.

\end{remark}

The above theorem also holds in higher dimensions and in more generality serving as a means of constructing geometric orbifolds from geometric building blocks of many sorts. For our purposes however, we want the result to be a hyperbolic manifold rather than an orbifold, this is equivalent to requiring the lattice be torsion free.
Fortunately, in dimension $3$ we get the extra information that the lattice is torsion free as torsion elements are ruled out by a combination of the angle and orientation equations. Indeed, the orientation restriction ensures that all elements of the lattice are orientation preserving, ensuring that the faces cannot contain singular points and the angle equations ensure that interior points of edges cannot be singular points, thus the only points that can be singular are vertices. However, orientation preserving point stabilisers (elements of $SO(3)$) are rotations about some axis, and so this axis has to form part of the singular locus of the manifold thus guaranteeing that the singular set cannot be $0$-dimensional, if it exists. Thus the only conclusion is that the singular locus is empty and the geometric orbifold is in fact a geometric manifold, as desired.

We finish by linking this result back into the results of Section \ref{SecPolynomial}.
Define $\Sigma$ to be a system of equations where the variables are given by the vertices of the tetrahedra in $\mathcal{M}$ and the equations are the edge equations, angle equations and equations which are equivalent to the orientation hypothesis. We will soon construct such a $\Sigma$ more explicitly.

\begin{corollary}\label{hypstructiffsoln}
The triangulation of $M$ admits a hyperbolic structure iff  $\Sigma$ admits a solution. Furthermore, each solution to $\Sigma$ encodes a hyperbolic structure of the triangulation $M$.
\end{corollary}

It's worthwhile here to note that all different hyperbolic structures on a triangulation coming from all these different solutions give the same hyperbolic structure on the manifold, this follows from Mostow Rigidity, which we now state for reference.

\begin{theorem}[Mostow-Prasad Rigidity \cite{Prasad73}]\label{MostowRigidity}
Let $\Gamma, \Gamma'$ be two lattices in $Isom(\h^n)$ and $\phi : \Gamma \rightarrow \Gamma'$ an isomorphism .
Then there exists an isometry $g$ in $Isom(\h^n)$ such that 
\[ g\gamma g^{-1} = \phi(\gamma) \quad \forall \gamma \in \Gamma .\]

In particular for closed hyperbolic $n$-manifolds $M, N$, $\pi_1(M) \cong \pi_1(N)$ if and only if $M$ and $N$ are isometric.

\end{theorem}

\section{The Existence of Geometric Triangulations} \label{SecGeometricTriangulations}

In this section we show how to verify whether a given triangulation admits a geometric structure making it a geometric triangulation.

The algorithm to do this is essentially a combination of Sections \ref{SecPolynomial} and \ref{SecPolyhedron} and consists of developing a system of polynomials which encodes the edge, angle and orientation equations and then applying Corollary \ref{hypstructiffsoln} Grigoriev's algorithm to that system.

For the rest of the section we will assume we are given a $3$-manifold $M$ admitting a hyperbolic structure, and a triangulation by $T$ tetrahedra which admits a geometric structure. For each tetrahedron in the triangulation, we model its vertices as points in the upper half space model for $\h^3$, that's $3$ variables per vertex, $4$ per tetrahedron, for a total of $12T$ variables, and we attempt to construct a system of polynomials which has all geometric structure(s) for this triangulation in its solution set. 
We shall track the complexity of the system throughout this section, but for the sake of brevity, we will only track this up to multiplication by a constant as the theorems of Section \ref{SecPolynomial} will absorb this constant into big $O$ notation. 

\begin{remark} 
Note that we use the variable $T$ to represent the number of tetrahedra in the triangulation, rather than the $t$ used so far. This is because later we shall be attempting to find geometric structures on quotients of subdivisions of some original triangulation, and shall denote the number of tetrahedra in the original triangulation by $t$ and the number in the new triangulation by $T$, in this case $T >> t$ and so it is useful to keep the variables separate. 
\end{remark}

\begin{remark} \label{RemarkSquareRoot}
Note that if we pay attention to how it affects numbers of polynomials/variables we can accept equations which involve both quotients and square roots of polynomials.
For example if $P, Q$ are polynomials, then the formula 
\[ \frac{P}{\sqrt{Q}} = 1 \]
would appear to cause us some problems. However this can be fixed by introducing a new variable $v_{\sqrt{Q}}$ such that 
\[ (v_{\sqrt{Q}})^2 = Q, \qquad v_{\sqrt{Q}} \geq 0.\]
We also fix the fact that it is a quotient by multiplying through by the denominator to get that the original equation is equivalent to the polynomial
\[ P - v_{\sqrt{Q}} =0.\] 
\end{remark}

\begin{remark}
Throughout this section we summarise the contribution of each lemma to the complexity of our system of polynomials at the end of each proof for use later.
\end{remark}

Sometimes, when we are defining our system of polynomials, we shall make an arbitrary choice from a possibility space of bounded size, for example it might be easier to guess in advance in which quadrant of the complex plane a variable lies and encode this guess as polynomials in our system rather than have our algorithm determine this. In this example, this guess would lead to four new systems, one for each possible guess, but if the system without the guess had a solution then, at least one of our new systems must do too and any solution to any of the four new systems corresponds to a solution of the old one. At this point we can just run our algorithm as many times as there are possible choices (in this example four) and guarantee that we find a solution in one of the runs. 

Thus there are two factors in determining the complexity of our algorithm, one is how long Grigoriev's algorithm takes to run given a system of polynomials, the other is how many times we have to run it based on how many guesses we have made. 

If our plan is to represent the equations of Poincare's polyhedron theorem as integer polynomials then we run into a problem with the angle equations. The equation $\Sigma \theta_j = 2\pi $ where the $\theta_j$ are the angles round an edge fails in two respects, that $2 \pi$ is not a integer is already problem enough but a more subtle point is that if we are given the vertices of a tetrahedron as variables, then the angles of that tetrahedron will occur as trigonometric functions of polynomials rather than simply polynomials in these variables.

\begin{lemma} \label{LemmaMakingChoices}
Let $M$ be as defined at the start of this section and suppose we have variables corresponding to $cos (\theta) $ and $sin(\theta)$, for each $\theta$ a dihedral angle of our model tetrahedra.
For each edge of the triangulation, we can list a finite set of choices of restrictions on the angles round that edge, and having made such a choice we can encode the angle equations as polynomials. At least one sequence of such choices for all edges of the triangulation corresponds to a hyperbolic structure on the manifold.
\end{lemma}

\begin{proof}
Instead of looking at the sum of the angles, which we saw above to be problematic, we ask that $\Pi e^{i\theta} = 1 $. This is polynomial in $cos(\theta), sin (\theta)$ as  $e^{i\theta} =  cos(\theta) + i sin(\theta)$. In fact as $\Pi e^{i\theta} = 1 $ is actually a statement about complex numbers, we get two polynomials, one setting the Real part equal to $1$, the other setting the imaginary part equal to $0$, both have degree bounded by $6T$, as this is the total number of dihedral angles.

The choice of restrictions we now make is that  for each $\theta$, we set the pair $(cos(\theta),sin(\theta))$ to lie in a box of the form $[ j /  2 T, (j+1) / 2T] \times [ j' /  2 T, (j'+1) / 2T]$ for $j,j'$ integers in $[-2T, 2T)$. Knowing these bounds tells us that $e^{i\theta}$ lies in a square in $\C$ of edge length $1/2T$ and thus on a subarc of the circle of length $<1/ T  < 2\pi / 6T$. Thus as there are less than $6T$ (the total number of edges in all tetrahedra) dihedral angles round an edge in our triangulation, if we know which boxes our $e^{i\theta} $ live in, then we know the sum of the angles $\theta$ up to an error of less than $2\pi$, thus we can pick boxes such that the only possible multiple of $2\pi$ the sum can be is $2\pi$ itself.

The total number of possible boxes is $16T^2$ and we have to pick one for each of the $6T$ angles, so the possible number of distinct choices is at most $<({16T^2})^{6T}$. Note that this number is certainly not optimised, but we shall see that it doesn't affect the runtime of an algorithm for deciding whether a triangulation admits a geometric structure.

For use later we note that the total number of polynomials we would need here would be the $2$ product equations per edge of the triangulation and the $4$ equations per dihedral angle, defining the box in which it lies. In total, this is less than $6 \times 6T$ polynomials, and no new variables. The degree of each is less than $6T$, and the most complex coefficient would be $\frac{2T-1}{2T}$ which has complexity $log_2(2T(2T-1)+ 2) < 8T$.
\end{proof}

Thus we now only need to find a way to encode variables corresponding to $cos(\theta), sin (\theta)$. To do this we shall need to change models of hyperbolic space.

\begin{figure}[ht]
\centering
    \includegraphics[width = 0.3\linewidth]{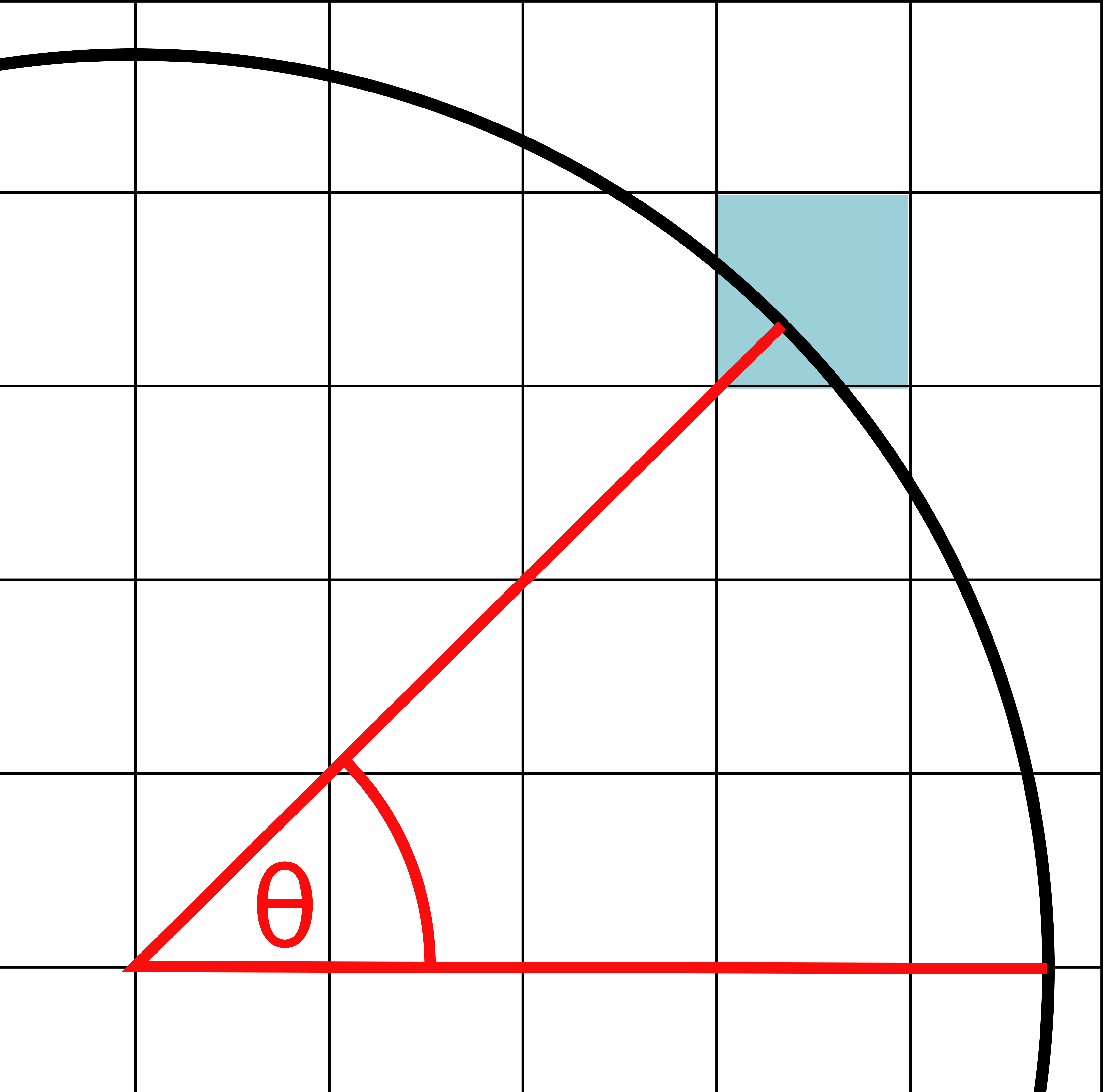}
    \caption{We can restrict the value of $\theta$ by requiring that the corresponding point on $S^1$ lies in a certain grid square.}
\end{figure}

\begin{lemma} \label{LemmaVariableCos}
Given variables for the vertices of a tetrahedron in the upper half space model and a chosen vertex, we can embed the tetrahedron isometrically in the Poincar\'e ball model such that the chosen vertex is at the origin. We then use this to introduce variables for each $cos(\theta)$ where $\theta$ is the dihedral angle between a pair of faces of the tetrahedron.
\end{lemma}

\begin{proof}
Suppose the $3$-simplex is given by the vertices $\mathbf{v}_0, \ldots , \mathbf{v}_3$ with the individual entries denoted by $\mathbf{v}_i =:  (v_{i,1}, \ldots, v_{i,3})$. First we define an isometry taking $\mathbf{v}_0$ to $(0,0,1)$ in the upper half space model. 
This isometry $\phi$ is defined by
\[(x_1, \ldots, x_3) \mapsto \frac{1}{v_{0,3}}(x_1 - v_{0,1}, x_{2} - v_{0,{2}}, x_3) \]
takes $\mathbf{v}_0$ to $(0,0,1)$.

We then map to the Poincar\'e ball model. The upper half space  and Poincar\'e ball models are related by the following isometry, $I$, which is in fact an involution of $\R^3 \cup \{ \infty \}$ but when restricted to either of our models, it sends each isometrically to the other:
\[(x_1, x_2, x_3) = \mathbf{x}  \mapsto 2\frac{\mathbf{x}+ e_3}{\langle \mathbf{x}+e_3, \mathbf{x}+e_3 \rangle} - e_3 \qquad \text{ where } e_3 = (0,0,1)\]

Define variables call them $\mathbf{v}_0', \ldots , \mathbf{v}_3'$ for the image of the vertices under the composition of these two maps (note that $\mathbf{v}_0'$ is the origin). The defining equation is a quotient of polynomials, so multiplying through by the denominator gives a polynomial equation defining the new variables. These polynomials have degree and coefficients bounded by a constant.

Suppose without loss of generality that we want to find the angle between the faces which intersect along the edge $\mathbf{v_0}' \mathbf{v_1}'$. Then because these faces intersect at the origin, they both lie in Euclidean planes through the origin. Thus the angle is given by the  formula for the angle between two Euclidean planes, so up to permuting the variables and changing the sign, the formula for each angle looks like

\[cos(\theta)  = \frac{\mathbf{v_1}' \times \mathbf{v_2}'}{\lVert \mathbf{v_1}' \times \mathbf{v_2}' \rVert} 
\cdot
\frac{\mathbf{v_1}' \times \mathbf{v_3}'}{\lVert \mathbf{v_1}' \times \mathbf{v_3}' \rVert}  \]

Note that this equation is a quotient involving square roots of polynomials, thus by the remark at the start of this section, each such equation can be represented by new polynomials  and variables whose number is bounded by a constant. 

If we perform this for every angle of our given tetrahedron and for every choice of which vertex goes to the origin in the Poincar\'e ball, we will still have introduced a number of polynomials and variables with degree and coefficient complexity all bounded by a constant. And thus to perform this for the whole triangulation we get variables for $cos(\theta)$ by introducing polynomials and variables whose number is bounded by a constant multiple of $T$.

All that is left is to note that if (as we will soon show) we know all the tetrahedra to be correctly oriented, then the angles we are talking about lie between $0$ and $\pi$ and thus $sin(\theta)$ is non-negative. Thus the following definition of the variable $sin(\theta)$ is sufficient
\[ sin(\theta)^2 = 1 - cos(\theta)^2 , \qquad sin (\theta) \geq 0 . \]

\end{proof}

Another similar but much easier to deal with problem is that if we're modelling a collection of geometric tetrahedra by their vertices in the upper half space model of hyperbolic space, then distance between vertices is once again not a polynomial function of the vertex variables. The formula is 
\[ d(\textbf{x}, \textbf{y}) = 
ln\left( \left( \frac{\sqrt{\sum_{i=1}^{n} (x_i - y_i)^2} + \sqrt{(\sum_{i=1}^{n-1} (x_i - y_i)^2) + (x_n + y_n)^2}}{2\sqrt{x_ny_n}} \right)^2\right)\]

\begin{lemma}
We can encode the edge equations by polynomials in variables corresponding to exponentials of the length of the edges.
\end{lemma}

\begin{proof}
We take $E_{\textbf{x},\textbf{y}} = e^{d(\textbf{x},\textbf{y})}$ as our variable for edge length as it is defined by the polynomials
\[  E_{\textbf{x}, \textbf{y}} 4x_3y_3 =\left( \Sigma_1 + \Sigma_2 \right)^2,  \qquad E_{\textbf{x}, \textbf{y}} > 1 \]
Where $\Sigma_1, \Sigma_2$ are defined by 

\[ \Sigma_1^2 = (x_1-y_1)^2 + (x_2-y_2)^2 + (x_3-y_3)^2 , \qquad \Sigma_1 \geq 0 \]
\[ \Sigma_2^2 =  (x_1-y_1)^2 + (x_2-y_2)^2 + (x_3+y_3)^2 ,  \qquad \Sigma_2 \geq 0 \]

Note the requirement that $E_{\textbf{x}, \textbf{y}} >1$ implies that all edges will have non-zero length. This also provides an explicit example of Remark \ref{RemarkSquareRoot} .

If two edges $e,e'$ with vertices $\textbf{x},\textbf{y} , \textbf{x}',\textbf{y}'$ respectively are to be identified by our pairings then 
\[ l(e) = l(e') \Leftrightarrow E_{\textbf{x},\textbf{y}} = E_{\textbf{x}',\textbf{y}'}\]
where $l(e)$ is the length of $e$. Each of the $< 6T$ edges $e$ is identified with less than $6T$ other edges, so in total we need $<(6T)^2/2$ edge equations, plus six equations defining each $E_{\textbf{x},\textbf{y}}$, so in total the number of polynomials is bounded by a constant multiple of $T^2$ and the number of new variables is bounded by a constant multiple of $T$, the degree and the complexity coefficients are bounded by a constant.
\end{proof}

\begin{lemma}
We can encode the requirement that our model tetrahedra be positively oriented by polynomials which also ensure that the tetrahedra are non-degenerate.
\end{lemma}

\begin{proof}
As we defined in the proof of Lemma \ref{LemmaVariableCos} we have variables corresponding to a normal vector to each face, we can then check that the dot product of the vertex opposing this face with this normal vector is positive. If we do this for every face we also handily get a guarantee of the non-degeneracy of each tetrahedron. 

This only takes four such polynomials per tetrahedron and thus we again get a constant multiple of $T$ polynomials with degree and coefficient complexity bounded by a constant.
\end{proof}

Note that, as we have been doing so far, we can always add variables corresponding to some polynomial function of our initial variables. One reason to do this might be to use the results of Secton \ref{SecPolynomial} to find bounds on the quantities represented by these variables, as we shall soon do with edge lengths. 

%Another trick we shall use later is that we could run this algorithm on a pair of geometric triangulations simultaneously by having two disjoint sets of variables and polynomials which don't interact, while only effectively doubling the number of polynomials and variables. We could then introduce polynomials into this new system which require that our solutions be related to one another.
%We shall later use this to ask whether specific lattices generated by Poincar\'e's Polyhedron Theorem surject onto other ones, and thus it will be useful to have variables corresponding to generating sets of these lattices.

\begin{lemma}
We can define variables in our system which correspond to the generating elements of the lattice defined in Theorem \ref{PoincarePolyhedronTheorem1}.
\end{lemma}

\begin{proof}
As noted in Remark \ref{LatticeGenTechnicality} the orientation preserving isometry realising our face pairing is uniquely determined by the vertices of the two faces (note this requires the faces be non-degenerate).
However, before we can define these elements, it's easier to transfer to the hyperboloid model, where the image of a vertex under an element of $PSL(2,\C)$ has an easy explicit formula.

An isometry between the disc model and the hyperboloid model is given by

\[ J : (x_1, x_2, x_3) \mapsto \frac{(2x_1,2x_2,2x_3, 1 + \Sigma x_i^2)}{1-\Sigma x_i^2}  \]

By the same methods as above we can define variables corresponding to the images of our original vertex variables under the composition of the isometry $I$ defined earlier (from the upper half plane to the ball model), and the isometry $J$ (from the ball to the hyperboloid). These are defined by polynomials of degree and coefficient complexity bounded by a constant, and the number of variables and polynomials we need to introduce is bounded by a constant multiple of $T$.

Using these new hyperboloid vertex variables, we can define variables corresponding to each of these face pairing isometries as matrices in $SL(2,\C)$. The way an element $A$ of $SL(2,\C)$ acts on a point $\textbf{x} = (x,y,z,t)$ in the hyperboloid model is given by associating to $\textbf{x}$ the matrix

\[X  =  \begin{pmatrix}
t+z & x - iy \\
x + iy & t-z
\end{pmatrix}\] 
and then $A$ acts by

\[ X \mapsto AX A^*.\]
Thus if a vertex $\textbf{x}$ is mapped to $\textbf{y}$  and $X, Y$ are their associated matrices we define the face pairing by a collection of polynomials equivalent to the matrix multiplication
\[ Y = AX A^* \] for each vertex of the chosen face.

It then follows from Theorem \ref{PoincarePolyhedronTheorem1} and Remark \ref{LatticeGenTechnicality} that a generating set for the lattice is given by a collection of products of at most $T$ of of these matrices. In summary, introducing these new variables requires a number of variables and polynomials bounded by a constant multiple of $T$, the polynomials have their degree bounded by a constant or $T$ and their coefficient complexity is bounded by a constant.
\end{proof}
Having built our system of polynomials assuming we already have a chosen guess for which boxes the angles lie in, the following result follows by checking the complexity numbers at the end of each of the proofs in this section.

\begin{lemma} \label{LemmaSystemGeomTriang}
For each choice described in Lemma \ref{LemmaMakingChoices}, we can build a system of polynomials which  encodes the edge equations, the angle equations and the orientation hypothesis as defined in Section \ref{SecPolyhedron} for a simplicial complex $M$. This system also has variables corresponding to exponentials of edge lengths and to a generating set for the lattice described in Theorem \ref{PoincarePolyhedronTheorem1}.
The complexity of this system is defined by the following for some universal constant $C$:
\begin{itemize}

\item $\kappa \leq CT^2$
\item $N \leq CT$
\item $d \leq CT$
\item $M \leq CT$
 
\end{itemize}
 
\end{lemma}

The results of Section \ref{SecPolynomial} and Corollary \ref{hypstructiffsoln} now give us this immediate corollary.

\begin{corollary} \label{HyperbolicStructureAlgorithm}
Let $M$ be a $3$-manifold triangulated by $T$ tetrahedra. Then  if $M$ admits the structure of geometric triangulation of a hyperbolic $3$-manifold as described in Section \ref{SecPolyhedron}, we can find such a structure in time bounded by
\[T^{O(T^2)}. \]

\end{corollary}

\begin{proof}
If we run the algorithm of Theorem \ref{Grigoriev} on all the systems corresponding to all the choices described in Lemma \ref{LemmaMakingChoices}, then we get a solution for at least one of the choices if and only if $M$ is a hyperbolic manifold and the triangulation admits a geometric structure. This follows from Corollary \ref{hypstructiffsoln} and the fact that the list of choices exhausts all possibilities for the dihedral angles of the tetrahedra. 
The runtime of each algorithm and the number of times we must run it are both bounded by 

\[ T^{O(T^2)}.\]
This follows from applying Theorem \ref{Grigoriev} to Lemma \ref{LemmaSystemGeomTriang} and from the count of choices in the proof of Lemma \ref{LemmaMakingChoices}. Thus this is also a bound for the runtime of the entire algorithm after simplifying big O notation.
\end{proof}

\begin{corollary}
Given a triangulated $3$-manifold $M$ (by $T$ tetrahedra) and a collection of simplicial identifications of the form $\Delta \rightarrow \Delta'$ for $\Delta, \Delta'$ tetrahedra in $M$, we can decide whether those identifications define a simplicial quotient $Q$, whether that quotient is indeed a manifold and whether it admits a hyperbolic structure, all in time bounded by

\[T^{O(T^2)}. \]

\end{corollary}

\begin{proof}
That we can find a hyperbolic structure for a triangulated $3$-manifold in time bounded as above is given by Corollary \ref{HyperbolicStructureAlgorithm}. That we can check in polynomial time whether a $3$-dimensional gluing is a $3$-manifold is one of the `simple algorithms' described in \cite{FomenkoMatveev97}. Thus it only remains to show that we can check whether a collection of identifications defines a simplicial quotient and that we can do so in time bounded as above.

This can indeed be done. Note that the number of possible identifications between two tetrahedra is $\lvert A_4 \rvert = 12$ and the number of possible pairs of tetrahedra is $ {T \choose 2}$ and the number of possible equivalence classes of tetrahedra after identification is $<T$ and so the size of the data defining the quotient is polynomial in $T$. First we check, for each equivalence class, whether the identifications induce any self identifications of tetrahedra in the quotient and then check that no three tetrahedra now meet at a face in the quotient. Both these steps are polynomial in the size of the data defining the quotient and hence polynomial in $T$.
\end{proof}

\begin{corollary} \label{ListOfGeometricTriangulations}
We can list all possible oriented geometrically triangulated simplicial quotients of a given triangulation of a hyperbolic $3$-manifold consisting of $T$ tetrahedra in time bounded by 

\[T^{O(T^2)}. \]
\end{corollary}

\begin{proof}
As shown above, the number of possible identifications between two tetrahedra is $\lvert A_4 \rvert = 12$ and the number of possible pairs of tetrahedra is $ {T \choose 2}$. Thus there are less than 
\[2^{12 {T \choose 2}}\]
possible combinations of identifications.
For each we check whether it gives a geometrically triangulated simplicial quotient in time bounded by 
\[T^{O(T^2)}.\]
Thus the time taken to run this algorithm for all of the possible combinations of all the possible identifications is 

\[ 2^{12 {T \choose 2}} T^{O(T^2)} \] 
which simplifies to

\[ T^{O(T^2)}. \]

At this point, as checking whether a triangulated manifold is oriented is decidable in polynomial time \cite{FomenkoMatveev97}, we can also find the oriented sublist in the same time bound.

\end{proof}

Finally, for reasons that will become clear later, we will want to restrict ourselves only to the elements of the list for which the quotient map is degree one.

\begin{lemma} \label{DegreeOneList}
Given a simplicial quotient $\phi: M \rightarrow Q$ of oriented manifolds, we can find the degree of $\phi$ in polynomial time.
\end{lemma}

\begin{proof}
As shown above, a simplicial quotient $M \rightarrow Q$ can be defined by a collection of simplicial identifications between tetrahedra in $M$. Also note that these collections are such that any induced self identifications are trivial. Thus if we have an orientation on $M$, and we pick some tetrahedron $\Delta$ in $M$, then we can consider the collection of tetrahedra which are identified with $\Delta$ and ask whether that identification is orientation preserving or reversing in each case. The degree of the map is then the number of tetrahedra glued in an orientation preserving manner (including $\Delta$ itself) minus the number glued in an orientation reversing manner. This can be checked in time polynomial in the number of tetrahedra in $M$.

\end{proof}

We summarise the results of this section in the following Theorem.

\begin{theorem} 
Given a hyperbolic 3-manifold $M$ triangulated by $t$ tetrahedra, we can list all oriented geometrically triangulated simplicial quotients of $M$ where the quotient map is degree one in time bounded by
\[ T^{O(T^2)}. \]

\end{theorem}

\section{Pachner Moves and Subdivisions} \label{SecPachner}

In this section we shall introduce Pachner moves alongside associated definitions and show how they can be used to get from a triangulation to its subdivision.

\begin{defn}[Combinatorial Equivalence]
We declare two simplicial complexes to be combinatorially equivalent if there exist subdivisions of each which are simplicially isomorphic.
\end{defn}

\begin{defn}[Combinatorial $n$-Manifold]
A combinatorial $n$-manifold is a simplicial complex such that the link of every vertex is combinatorially equivalent to the boundary of the standard $n$-simplex.

\end{defn}

\begin{defn}[PL Manifold]
A manifold equipped with a homeomorphism to a  combinatorial manifold is called a PL manifold.
A PL or Piecewise Linear homeomorphism is one which induces a homeomorphism between combinatorial $n$-manifolds which is simplicial with respect to suitable subdivisions.
\end{defn}

\begin{defn}[Pachner Move]

Given a combinatorial $n$-manifold $K$, and some combinatorial $n$-disc $D$ which is a subcomplex both of $K$ and of the boundary of the $n+1$-simplex $\Delta$, then a Pachner move consists of replacing $D$ by $\partial \Delta - D$.

\end{defn}

A Pachner move is defined by the number of simplices which make up $D$, and so it makes sense to refer to a $(1 \text{-} 4)$ move for example which replaces one 3-simplex with four 3-simplices.
These moves allow us to move between different triangulations of the same manifold. In fact, Pachner proved the following equivalence.

\begin{theorem}[Theorem 5.5 \cite{Pachner91}]
Two closed combinatorial $n$-manifolds are PL homeomorphic iff they are related by a sequence of Pachner moves and
simplicial isomorphisms.
\end{theorem}

The results of this section use the key insight of Kalelkar and Phanse \cite{KalelkarPhanse19} to apply Theorem \ref{AdiprasitoBenedetti}, due to Adiprasito and Benedetti \cite{Adiprasito17} in conjunction with Theorem \ref{Lickorish}, due to Lickorish \cite{Lickorish99}.

\begin{theorem}[Theorem A - \cite{Adiprasito17}] \label{AdiprasitoBenedetti}
If C is any (polytopal) subdivision of a convex polytope, the second barycentric subdivision of C is shellable. If dim(C) = 3, already the first barycentric subdivision of C is shellable.
\end{theorem}

\begin{defn}[Polytopal Subdivision]

A polytope P in $\R^n$ is the convex hull of finitely many points. A face $F$ of $P$ is either P itself or $P \cap H$ where $H$ is some hyperplane such that it intersects $P$ non-trivially and $P$ lies entirely in one of the half spaces defined by $H$.
A polytopal complex is a collection of polytopes in $\R^n$ which is closed under taking faces and such that two polytopes intersect in a common face or not at all.
A polytopal subdivision $C'$ of a polytopal complex $C$ is a polytopal complex which has the same underlying space as $C$ and such that every face of $C'$ lies in some face of $C$.

\end{defn}

\begin{defn}[Shellability]
A combinatorial $n$-ball is shellable if there exists an ordering of the $n$-simplices such that we can remove all but the final $n$-simplex in order, at each stage being left with a connected combinatorial $n$-ball.
\end{defn}

\begin{theorem}[Lemma 5.7 \cite{Lickorish99}]\label{Lickorish}
If $X$ is a shellable combinatorial n-ball made up of $k$ n-simplices, then $X$ can be transformed into a cone on its boundary by a sequence of $k$ Pachner moves.
\end{theorem}

\begin{proof}

We use induction on $k$. If $k = 1$ then the (1-$(n+1)$) Pachner move suffices.
Suppose the statement is true for all balls with fewer than $k$ $n$-simplices, then number the $n$-simplices of $X$ according to some shelling and consider the $k-1$ simplices labelled $2, \ldots , k$. By shellability, these form a combinatorial $n$-ball themselves and so by induction we can perform $k-1$ Pachner moves converting this ball to a cone on its boundary. At this point, the simplex labelled $1$ shares a face with $j \leq n$ simplices of this ball, and performing a ($(j+1)$ - $(n+1-j)$) move respectively leaves us with a cone on $\partial X$. 

\end{proof}

\begin{figure}[ht]
\centering
    \includegraphics[width = 0.8\linewidth]{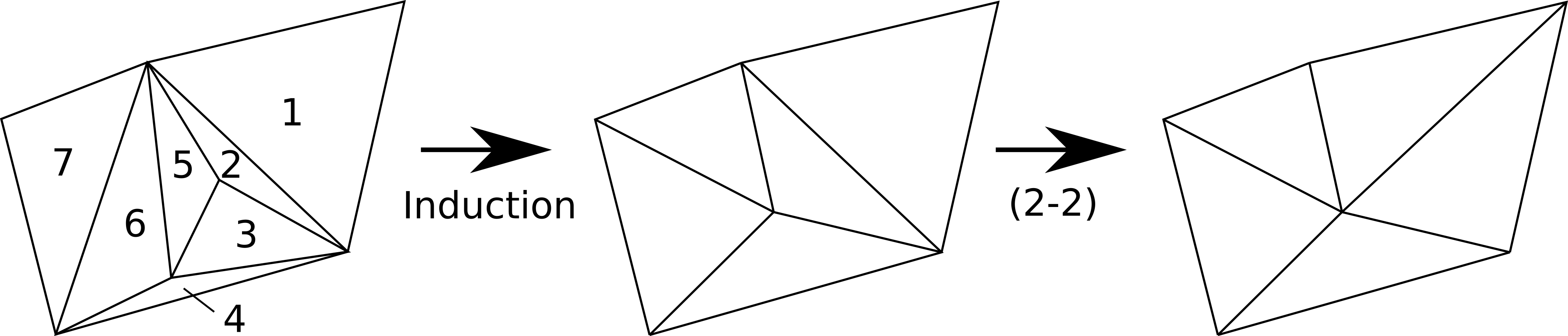}
    \caption{A particularly simple example of coning a shelling.}
\end{figure}

%\begin{remark}
%Suppose we have a triangulation $\T_1$ by $t$ tetrahedra and a subdivision $\T_2$ by $T \geq t$ tetrahedra such that the subdivision induced on each tetrahedron of $\T_1$ is shellable. The subdivision $\T_2$ induces a subdivision of the $2$-skeleton of $\T_1$. If we take this subdivision of the $2$-skeleton and simply subdivide each tetrahedron by coning to its boundary, we get a new subdivision $\T_2'$ which agrees with $\T_2$ on the $2$-skeleton of $\T_1$ but induces a cone structure on each tetrahedron of $\T_1$. Thus, using Theorem \ref{Lickorish}, we can perform Pachner moves taking $\T_2$ to $\T_2'$. Furthermore, at each step we still have a subdivision of $\T_1$. This is because these moves preserve the boundary of each original tetrahedron from $\T_1$.
%
%Theorem \ref{AdiprasitoBenedetti} tells us that  the barycentric subdivision of any given subdivision satisfies the above shellability condition. Also barycentric subdivision simply multiplies the number of tetrahedra in a triangulation by $24$, so makes no difference to the complexity of the algorithms described in this paper. Thus from now on we shall assume that our subdivisions $\T_2$ induce shellable triangulations on the tetrahedra of $\T_1$.
%\end{remark}

We shall now define composite moves (finite combinations of Pachner moves) which preserve this cone structure. These composite moves allow us to get from $\T_1$ to $\T_2'$ while again ensuring that we never lose the structure of being a subdivision of $\T_1$. This is important as it will mean that whatever moves we choose in whatever order, the result will always be a subdivision. We shall later use this to list subdivisions by performing arbitrary sequences of moves.

\begin{defn}[Coned Subdivisions and Preserving the Cone Structure]
Coned subdivisions are defined to be those subdivisions of a simplicial $n$-complex which are formed by subdividing the $(n-1)$-skeleton and then taking the subdivision on each $n$-simplex to be a cone to its boundary.
Moves which preserve the property of being a coned subdivision shall be said to preserve the cone structure.
\end{defn}

\begin{remark}
A coned subdivision of a triangulated 3-manifold is uniquely determined by the subdivision it induces on the $2$-skeleton of the original triangulation. Thus it makes sense to refer to \textbf{the} coned subdivision which induces a certain subdivision of the $2$-skeleton.
\end{remark}

What gives us hope that it should be possible to make moves between coned subdivisions is that if we have a cone structure on each side of a subdivided $2$-simplex, this means we have an embedded suspension of this subdivided $2$-simplex, and we can perform $2$-dimensional Pachner moves on the subdivision of this $2$-simplex by performing $3$-dimensional moves on the suspension as shown in the following.

\begin{lemma}\label{SuspendedPachnerMoves}
Let $\Delta_1, \ldots \Delta_k$ be 2-simplices arranged so that it is possible to perform a 2-dimensional Pachner move, (k= 1, 2 or 3 depending on the move). Then this move can be realised in the suspension $S(\Delta_1 \cup  \ldots \cup \Delta_k$) by performing two 3-dimensional Pachner moves.
We shall refer to these composite moves as two-dimensional moves.
\end{lemma}

\begin{proof}$ $\\
\textbf{The (1-3) Move}\\
Perform a (1-4) move on one of the tetrahedra, and then perform a (2-3) move incorporating the other tetrahedron.\\
\textbf{The (2-2) Move}\\
Perform a (2-3) move on the tetrahedra on one side of the suspension, then a (3-2) move incorporating the tetrahedra on the other side.\\
\textbf{The (3-1) Move}\\
Perform a (3-2) move on the tetrahedra on one side of the suspension, then a (4-1) move incorporating the tetrahedra on the other side. Note that this is simply the reverse of the (1-3) case.

\end{proof}

\begin{figure}[ht] 
\centering
    \includegraphics[width = 0.9\linewidth]{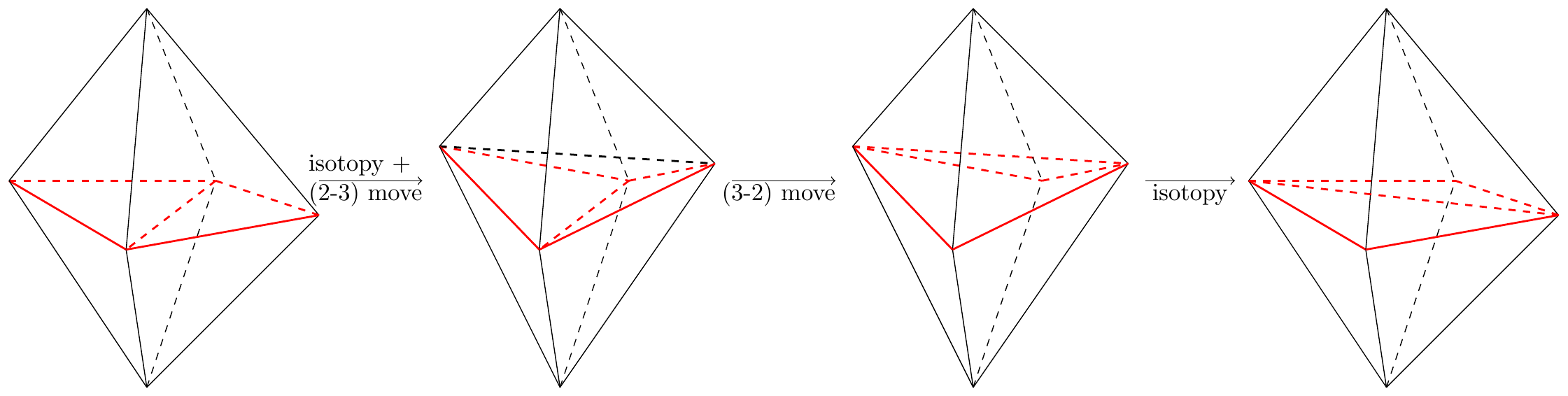}
    \caption{Performing the (2-2) move in suspension.}
\end{figure}

\begin{figure}[ht] 
\centering
    \includegraphics[width = 0.9\linewidth]{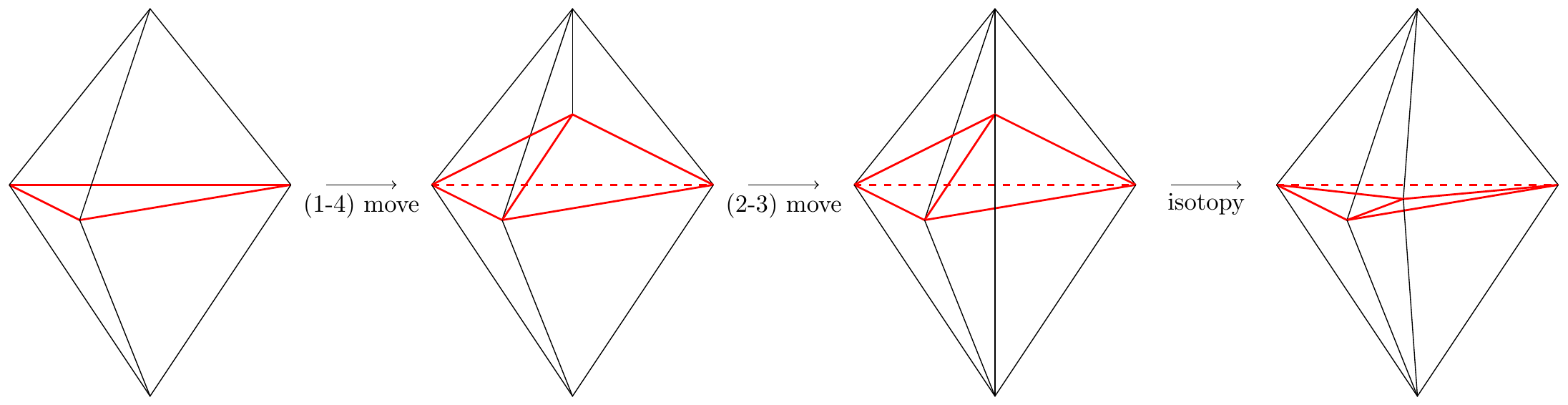}
    \caption{Performing the (1-3) move in suspension. Note the (3-1) move can be done by simply performing these moves in reverse.}
\end{figure}

\begin{defn}[First Coned Subdivision]
Two-dimensional moves can only be performed in suspension, so we need an initial cone structure.
Our first step is to perform a $(1$-$4)$ move on each tetrahedron of $\T_1$ and then perform $(1$-$3)$ moves on each original face of $\T_1$. Call this the first coned subdivision of $\T_1$.
\end{defn}
We now show how to subdivide the $1$-skeleton of this coned subdivision.

\begin{lemma}
Let $\T$ be the first coned subdivision of $\T_1$, then we can add as many vertices as we like to an edge of $\T$ coming from $\T_1$ while preserving the cone structure.
\end{lemma}

\begin{proof}
Consider the star of a given edge $e$, coming from $\T_1$. The worst case scenario is that $e$ was an edge of every tetrahedron of $\T_1$ and so its star in $\T$ contains $12t$ tetrahedra. Thus $st(e)$ is made up of $k \leq 12t$ tetrahedra labelled $T_1, \ldots T_k$ arranged anticlockwise around $e$.\\
Let us also label the vertices of $e$ (which are shared by all the $T_i$) by $u$ and $v$ and label by $w_i$ the vertex shared by $T_i$ and $T_{i+1}$ (mod $k$). %Note that each of the $w_i$ is a vertex added in stage 1. \\
Performing a (1-4) move on $T_1$ creates a new vertex, $x$, and 4 tetrahedra, only one of which shares a face with $T_2$ - the tetrahedron with vertices $(u, v, x, w_1)$. The goal is to make $x$ the new vertex bisecting $e$.\\
We can now perform a (2-3) move on the pair formed by $T_2$ and the tetrahedron $(u, v, x, w_1)$, which gives us 3 tetrahedra, each containing $x$, one of which shares a face with $T_3$. We can keep performing (2-3) moves like this, each time connecting $x$ to a new $w_i$ until $x$ is connected in this way to $w_{k-1}$. At this point, the vertex $x$ is connected to all the vertices of $st(e)$ but we still have the original edge $e$, which we want to eliminate. The three tetrahedra $T_k$, $(u, v, x, w_{k-1})$ and $(u, v, w_{k-1}, w_k)$ surround $e$ and a (3-2) move on these three removes $e$. Thus we are finished after $k \leq 12t$ Pachner moves.

Note that the $w_i$ are the coning vertices making $\T$ a coned subdivision, and the process described above ensures exactly that each of these vertices now also connects to the new vertex $x$, thus this entire composite move preserves the cone structure. The new edges we have created are such that their stars also contain $\leq 12t$ tetrahedra, thus performing this sequence for a new edge also requires $\leq 12t$ moves.

\end{proof}

\begin{figure}[ht]

\centering
\includegraphics[width = \linewidth]{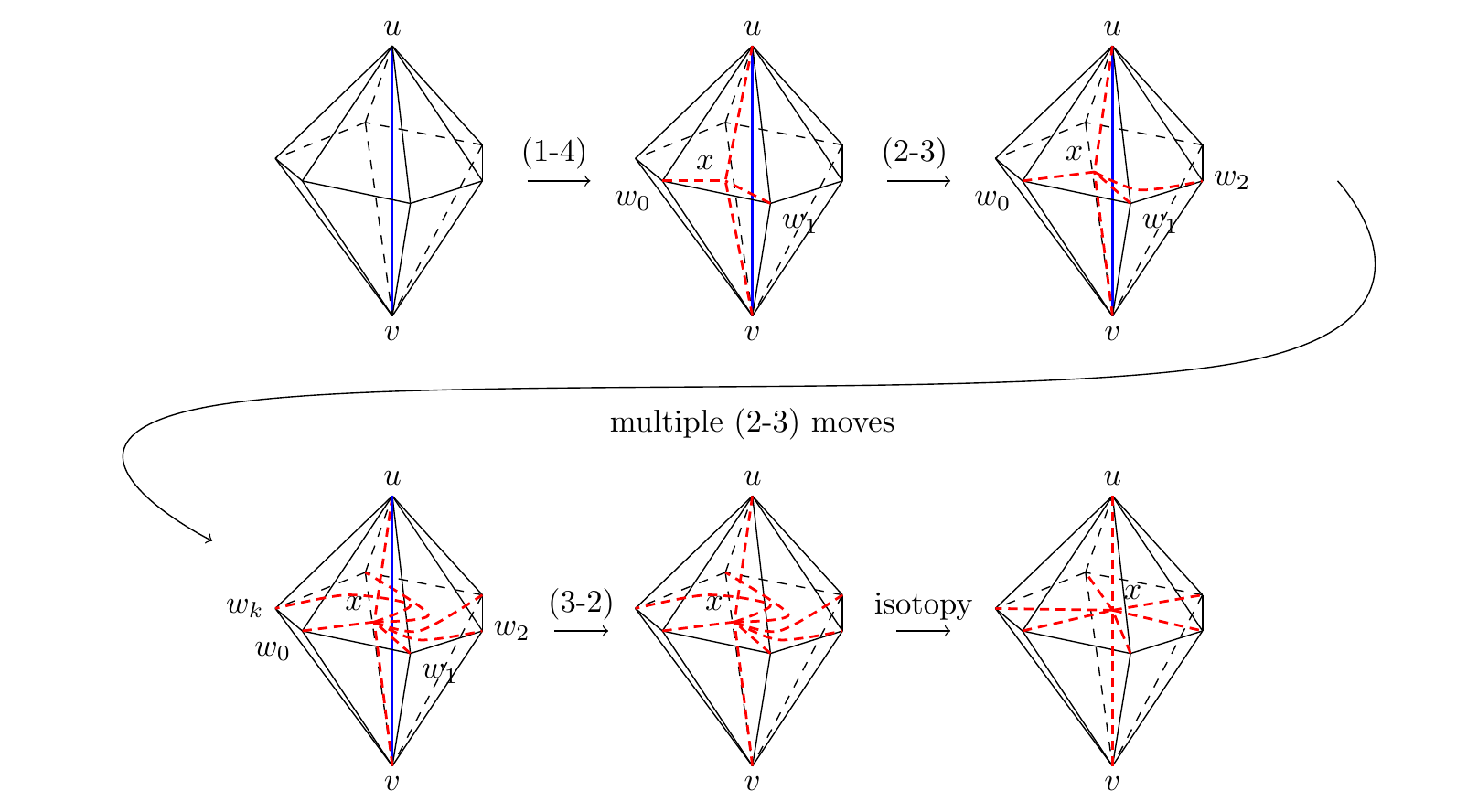}
    \caption{A sequence of Pachner moves inserting a vertex into an edge of our triangulation.}

\end{figure}

Noting that each of these moves (which we shall call vertex adding moves) consists of $\leq 2t$ Pachner moves, and using also the two-dimensional Pachner moves described above (which require two $3$-dimensional Pachner moves) we can now show how to get from this first coned subdivision to the coned subdivision $\T_2'$ defined earlier.

\begin{lemma}
Let $\T$ be the first coned subdivision of $\T_1$, and $\T_2'$ be the coned subdivision of $\T_1$ which agrees with $\T_2$ on the $2$-skeleton of $\T_1$.
Then we can get from $\T$ to $\T_2'$ through a sequence of combinatorial moves. We do this by a sequence of vertex adding moves followed by a sequence of two-dimensional Pachner moves on the $2$-simplices coming from $\T_1$. The number of each type of move required is bounded by $4T$. Recall $\T_1$ and $\T_2$ are triangulated by $t,T$ tetrahedra respectively with $T>t$.
\end{lemma}

\begin{proof}
As vertex adding moves and two-dimensional Pachner moves preserve the cone structure, and  coned subdivisions are uniquely determined by their restriction to the $2$-skeleton of $\T_1$, we only need to show that we can change the $2$-skeleton of $\T$ to that of $\T_2'$ using only vertex adding moves and two-dimensional moves.
We first perform vertex adding moves until the triangulation agrees with $\T_2'$ on the $1$-skeleton. The number of necessary vertex adding moves is bounded above by the number of vertices in $\T_2'$, which is bounded above by four times the number of $3$-simplices in $\T_2$, $4T$.

We then have a triangulation which agrees on the $1$-skeleton with $\T_2'$, and in fact when we restrict to the $2$-skeleton of $\T_1$, our triangulation is a coned subdivision of the $2$-skeleton of $\T_1$.

This means that if we look at the subdivision of each $2$-simplex of $\T_1$ induced by our subdivision so far, compared to the subdivision induced by $\T_2'$ the first is a cone to its boundary, and the second is a shellable triangulation of the $2$-simplex which agrees with the first on its boundary, and consists of  $c$ 2-simplices, for some constant $c$. Thus using Lemma \ref{Lickorish} we can perform $c$ two-dimensional Pachner moves after which the triangulation agrees with $\T_2'$ on this $2$-simplex. Now the total number of such moves we need to make is bounded above by the total number of $2$-simplices in $\T_2$ and thus is bounded by $4T$.

Thus, after performing the above process on each $2$-simplex, we have transformed $\T$ to $\T_2'$ in $\leq 4T$ vertex adding moves and $\leq 4T$ two-dimensional moves.
\end{proof}

It is now a simple corollary to get from $\T_1$ to $\T_2$.

\begin{theorem} \label{SubdivisionMoves}
Suppose we have a $3$-manifold with a triangulation $\T_1$ by $t$ tetrahedra and a subdivision $\T_2$ by $T$ tetrahedra such that the induced triangulation of $\T_2$ on each tetrahedron of $\T_1$ is shellable. Then we can relate $\T_1$ and  $\T_2$ by a sequence of composite moves (combinations of Pachner moves)  such that each composite move takes one subdivision of $\T_1$ to another. We can track after each such move the map from the $2$-skeleton of $\T_1$ to the barycentric subdivision of $\T_2$ induced by the subdivision.

The number of Pachner moves required to perform these composite moves is bounded by 

\[  48tT + 9T + 9t. \]

\end{theorem}

\begin{proof}
$\T_2$ consists of $T$ tetrahedra, and thus we can use Lemma \ref{Lickorish} on each tetrahedron $\Delta_i$ of $\T_1$ to change $\T_2$ to the coned subdivision $\T_2'$. $\T_2'$ is defined as earlier as the coned subdivision of $\T_1$ which agrees with $\T_2$ on the $2$-skeleton of $\T_1$. If $\T_2$ induces a subdivision by $T_i$ tetrahedra on each $\Delta_i$ then Lemma \ref{Lickorish} tells us it takes $T_i$ Pachner moves to change the subdivision of $\Delta_i$ to a coned subdivision. Thus it takes $T = \sum T_i$ Pachner moves to change from $\T_2$ to $\T_2'$ and thus the same number takes us in the reverse direction also.

Thus to get from $\T_1$ to $\T_2$, we first cone to get the first coned subdivision $ \T$ of $\T_1$, using $t$ Pachner moves to cone the tetrahedra, and $<4t$ two-dimensional moves to cone the $2$-simplices. From here it takes $< 4T$ vertex adding moves and $< 4T$ two-dimensional moves to get to $\T_2'$ and as we've just shown $ \leq T$ Pachner moves completes the transformation to $\T_2$.

As each move preserves the cone structure, we can track the subdivision of the $2$-skeleton of $\T_1$ throughout, and hence we have a map from the $2$-skeleton of $\T_1$ into that of the barycentric subdivision of $\T_2$.

To verify the count of Pachner moves, note that two-dimensional moves take two Pachner moves each and vertex adding moves take $\leq 12t$ Pachner moves each.
Thus the final total of Pachner moves applied is

\[ < t + 8t + 48 tT + 8T + T= 48tT + 9T + 9t \]
\end{proof}

\begin{remark} \label{RemarkShellabilityTechnicality}
Note that even if $\T_2$ doesn't induce shellable triangulations on the tetrahedra of $\T_1$, then its barycentric subdivision does and the barycentric subdivision of a triangulation by $T$ tetrahedra consists of $24T$ tetrahedra. Hence, the number of Pachner moves required to get from a triangulation $\T_1$ to the barycentric subdivision of a subdivision $\T_2$ is bounded above by 

\[ 1152tT + 216T +9t .\]

Note two things here, firstly this difference only changes the result by a constant and so will be absorbed into big $O$ notation later. Secondly, performing barycentric subdivisions preserves the property of admitting a geometric structure, and commutes with taking simplicial quotients.
Thus for the sake of simplicity we shall from this point omit mentioning where we might take a barycentric subdivision and pretend that all subdivisions we are dealing with already admit the desired shellability properties. 
\end{remark}

\section{A Geometric Triangulation of Bounded Complexity} \label{SecGeometricQuotient}
We now recall some results from previous work of the author \cite{Scull21}.

\begin{theorem}[Lemmas 3.9, 3.10, 5.1 and Theorem 3.17 of \cite{Scull21}]\label{InjRadSummary}
Let $M$ be a closed hyperbolic $n$-manifold $(n \geq 3)$ triangulated by $t$ $n$-simplices. Let $X$ be a hyperbolic manifold homeomorphic to $M$, then there is a homotopy equivalence 
\[ F : M \rightarrow X \]
such that the image of each $n$-simplex in $M$ is a geodesically immersed hyperbolic $n$-simplex in $X$. Furthermore $F$ is surjective and the length $l(e)$ of an edge in the image of $F$ is bounded as follows.

\[ l(e) \leq  (nt)^{O(n^4t)}.\] 

\end{theorem}

The following two theorems give control over the length of systoles and hence over the injectivity radius in closed hyperbolic $n$-manifolds.

\begin{theorem}[Theorem 1.1 of \cite{Scull21}] \label{TheoremInjRadBound}
Given a closed hyperbolic $n$-manifold $(n \geq 3)$ triangulated by $t$ $n$-simplices, the length $R$ of a systole (shortest closed geodesic) of $M$ is bounded below by a function of $t$ and $n$, in particular 
\[ R \geq \frac{1}{2^{(nt)^{O(n^4t)} }} .\]
\end{theorem}

\begin{theorem}[Theorem 4.8 \cite{Scull21}]\label{Scull20Paper}
Suppose that $M$ is a finite volume hyperbolic $n$-manifold  $(n \geq 3)$ with systole(s) of length $R \leq 2\epsilon_n$, where $\epsilon_n$ is the Margulis constant in dimension $n$. Then the distance from a systole to the boundary of the Margulis tube containing it is bounded below by \[\frac{1}{n} log\left(\frac{1}{R}\right) +  log (\epsilon_n) -log( 4). \]

\end{theorem}

The first theorem tells us that our original triangulation of $M$ admits a degree one map $M \rightarrow M$ such that its image consists of geodesic $n$-simplices which cover the manifold. We now attempt to use what we know of the injectivity radius of the manifold to form a subdivision of each of these geodesic simplices so that the subdivided cover has the property that any two  simplices can only intersect in one convex component.

For closed hyperbolic manifolds, this is as simple as subdividing so that the $n$-simplices have diameter less than the injectivity radius of the manifold. For this we use the above bound on injectivity radius. However, in three dimensions, a better bound exists by a combination of the following theorem of White (see also the stronger result of Agol and Liu \cite{AgolLiu10} and Theorem \ref{Scull20Paper}.

\begin{theorem}[Theorem 5.9, White \cite{White00a}]
There is an explicit constant $K>0$ such that if M is a closed, connected, hyperbolic 3-manifold, and 
\[ P= \langle x_1, \ldots, x_n \mid r_1, \ldots , r_m \rangle \]
 is a presentation of its fundamental group, then $diam(M) < K(l(P))$, where
 \[ l(P) =  \sum^m_{i=1} l(r_i) \]
 and $l(r_i)$ is the word length of a given relator.

\end{theorem}

Now, given a simplicial complex, we can pick some spanning tree of its 1-skeleton, then a generator set is given by the remaining edges of the 1-skeleton, and a relator set, each of which is a word of length at most 3, is given by the 2-simplices. This leads to the following corollary.

\begin{corollary}\label{CorollaryWhite}

Let $M$ be a closed, connected, hyperbolic 3-manifold triangulated by $t$ tetrahedra. The number of 2-simplices is exactly  $2t$ and so $l(P) \leq 6t$ where $P$ is the presentation described above. Hence, $diam(M) < 6Kt < O(t)$.

\end{corollary}

\begin{corollary}\label{3MfdInjRad}
Let $M$ be a closed, connected, hyperbolic 3-manifold triangulated by $t$ tetrahedra. The injectivity radius of $M$ is bounded below by 
\[ e^{-O(t)} \]
\end{corollary}

\begin{proof}
Either the injectivity radius of $M$ is bigger than $\epsilon_3$ in which case the result follows, or Theorem \ref{Scull20Paper} applies and by Corollary \ref{CorollaryWhite}
\[ log (\frac{1}{R}) < O(t)\]
from which the statement follows.
\end{proof}

\begin{remark}\label{RemarkInjRadEmbed}
Let $F: \Delta \rightarrow X$ be a geodesic immersion of a hyperbolic tetrahedron $\Delta$ and $\tilde{F} : \Delta \rightarrow \h^n$ be a lift of $F$. We note that for such a geodesic immersion to self intersect, the diameter of the lift of the tetrahedron must be greater than the length of a systole, or twice the injectivity radius of the hyperbolic manifold $X$.
\end{remark}

Our first step is to use this remark to find a subdivision of each simplex in $M$ (subdivisions may not agree on the boundary and so this might not be a subdivision of the whole triangulation) such that each of the new simplices embeds into $X$. The next two lemmas describe how to perform this subdivision.

\begin{lemma} \label{LemmaInductivePolyhedra}
Let $\Delta$ be a hyperbolic $k$-simplex with $k \leq n$ and edge lengths bounded above by $L>0$. Suppose we have a subdivision of one of its faces $F$ into $K$ hyperbolic  polyhedra, each of which has diameter bounded above by $\frac{k-1}{n}D$ for some $D > 0$ and has at most $2(k-1)$ faces. 
Then we can subdivide the $k$-simplex itself into $K\frac{Ln}{D}$ polyhedra, each of diameter bounded above by $\frac{k}{n}D$ with at most $2k$ faces.
\end{lemma}

\begin{proof}
Take an embedding of $\Delta$ in the Klein model of hyperbolic $k$-space such that the vertex $v$ opposite $F$ is at the origin. Next, cone the polyhedral decomposition of $F$ to the vertex $v$. 

Now, we insert parallel scaled copies of $F$ throughout the $k$-simplex such that each lies in the $\frac{D}{n}$-neighbourhood of the last. We do this by ensuring that, for every vertex of $F$, and each corresponding edge coning to $v$, the length of the edge segment $e_i$ between two copies $F_1$ and $F_2$  along this (hyperbolic geodesic) edge is less than $\frac{D}{n}$. This ensures that every point on $F_1$ is with a $\frac{D}{n}$-neighbourhood of $F_2$ as for any point the radial arc from $F_1$ to $F_2$ is both shorter in Euclidean length and has its vertices closer to the origin than the maximal length $e_i$, thus its hyperbolic length is also less than that of $e_i$, due to how the metric in the Klein model scales down as we go further from the origin. 

This gives us a polyhedral subdivision into $K\frac{Ln}{D}$ polyhedra each with at most $2k$ faces (each of which has (2k$-$2) faces and so on). To see that the diameter of each polyhedron is bounded above by $\frac{kD}{n}$, note that any two points in the polyhedron can be connected by a radial arc, followed by a path in a scaled copy of one of the polyhedra in $F$, thus this path has total length bounded above by the lengths of the two subpaths
\[\frac{D}{n} + \frac{(k-1)D}{n} = \frac{kD}{n}\]
\end{proof}

\begin{figure}[ht]
 \centering
    \includegraphics[width = 0.5\linewidth]{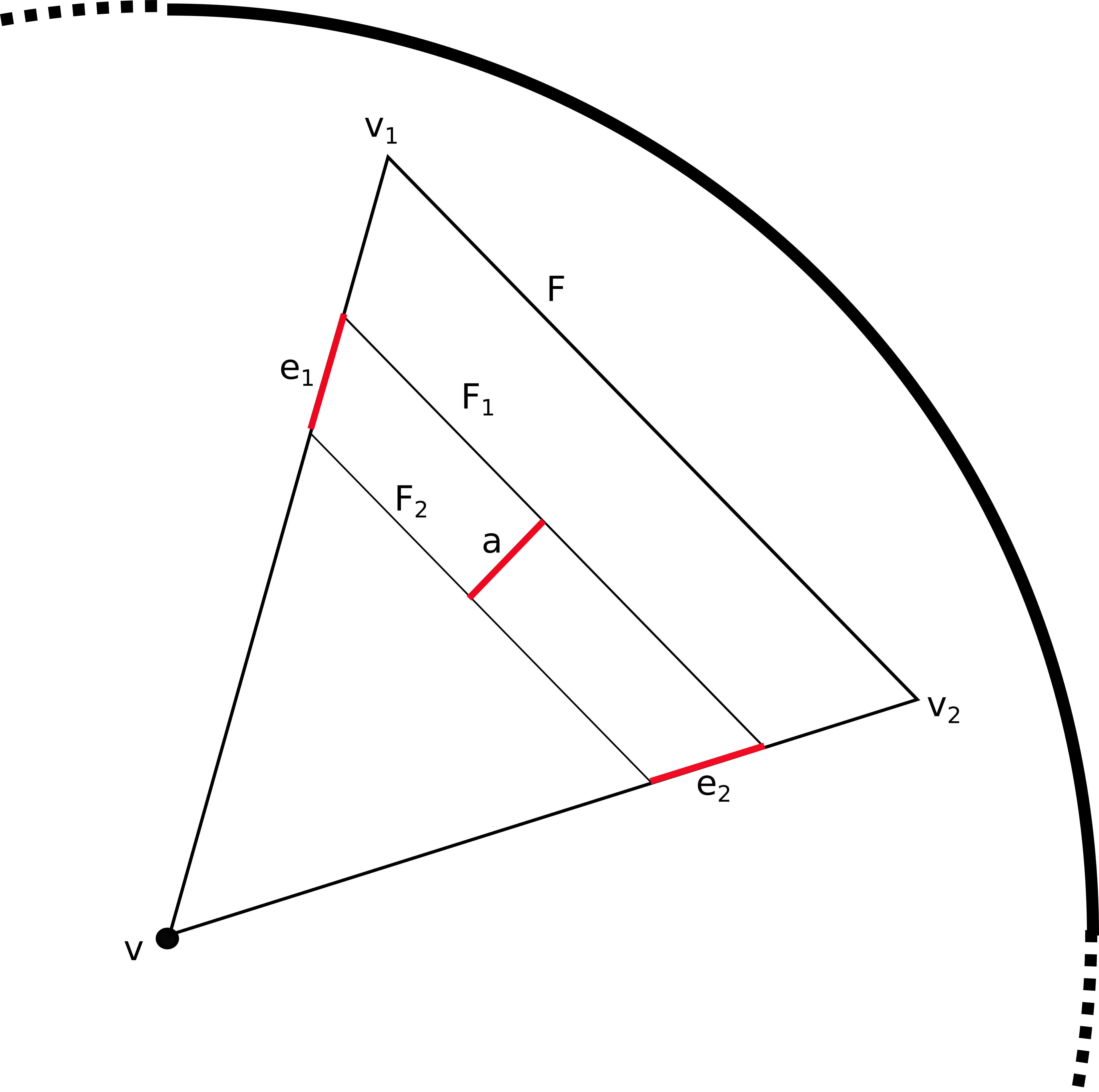}
    \caption{Let $a$ be a radial segment between two translates $F_1,F_2$ of $F$ and let $e_i$ denote the segment of the edge $v$-$v_i$ lying between $F_1,F_2$ for $v_i$ the vertices of $F$. If $v_i$ is the further vertex from $v$, then the angle made by $e_i$ with the translates $F_1,F_2$ is further from perpendicular than the angle made by $e$, thus it has has both greater Euclidean length and is further from the origin than $e$. Hence the hyperbolic length of $e_i$ is greater than the hyperbolic length of $e$, so if both segments have length bounded above by $D/n$, that is also a bound for the length of $e$. }
$ $\\
%    \raisebox{0.25\height}

\end{figure}

\begin{figure}[ht]
    {\includegraphics[width = 0.5\linewidth]{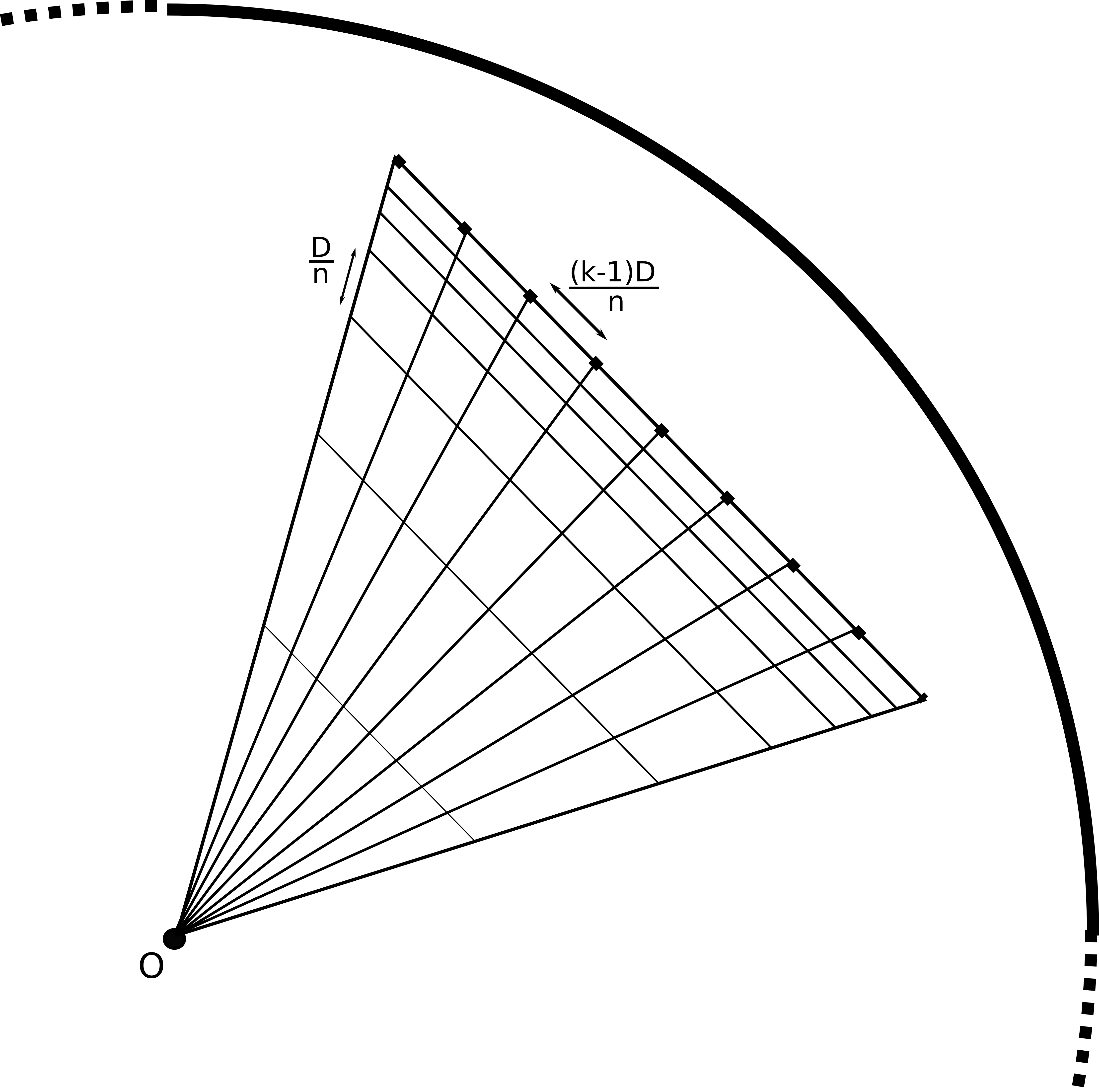}}
   
    \caption{Note that the radial segments lie a similar euclidean distance apart, while the translates of the exterior face bunch up as we go further from the origin. This is because in the Klein model, as distance from the origin increases smaller and smaller euclidean segments can measure the same hyperbolic distance.}
\end{figure}  

\begin{lemma}\label{LemmaEuclideanSubdivide}

The number of $n$-simplices required to triangulate a convex polyhedron formed by the intersection of $J$ halfspaces bounded by hyperplanes in $\R^n$ is bounded above by the following

\[\left( 2^{J} \right)^{n+1} \]

In particular if $J =2(n+1)$, then the number of $n$-simplices required is bounded by

\[ \left( 2^{2n+2} \right) ^{n+1} \leq 2 ^{7n^2}  \] 

\end{lemma}

\begin{proof}
We can triangulate the polyhedron by barycentrically subdividing it. Note that each $n$-simplex $\Delta$ of the barycentric subdivision is uniquely determined by a sequence $\sigma_0 \subseteq \sigma_1 \subseteq \ldots \sigma_{n}$ where each $\sigma_i$ is some $i$-dimensional face of the polyhedron where $\Delta$ is the unique $n$-simplex which intersects the interior of each face of the sequence. The number of faces of a given dimension is certainly less than $2^{J}$ as this is the size of the power set of the hyperplanes and every face is an intersection of hyperplanes. Thus the total number of  sequences of faces is bounded above by by 
\[ \left( 2^{J} \right) ^{n+1} \]     
\end{proof}

We now combine the above two lemmas to see how it relates to our original triangulation.

\begin{lemma} \label{GeometricImmersed}
Given $(M,\mathcal{T})$ a triangulated $n$-manifold homeomorphic to a hyperbolic n-manifold $X$, $F : M \rightarrow X$ the homotopy equivalence defined in Theorem \ref{InjRadSummary} and some desired bound $D<\frac{1}{n}$ on the diameter of $n$-simplices, we can find a subdivision of each simplex of $(M,\mathcal{T})$ such that $F$ restricted to each simplex has image a geodesically immersed hyperbolic $n$-simplex of diameter less than $D$.

Furthermore, the union of these subdivisions consists of less than the following number of $n$-simplices:
\[  (nt)^{O(n^4t)} \frac{1}{D^n}.\]
\end{lemma}

\begin{proof}

For each hyperbolic $n$-simplex with edge lengths bounded by $L>0$, we first inductively perform the process described in Lemma \ref{LemmaInductivePolyhedra} starting with the base case of a hyperbolic $1$-simplex. Clearly we can get a subdivision of a $1$-simplex of length at most $L$  into ``polyhedra'' of diameter less than $\frac{D}{n}$ which consists of $\frac{Ln}{D}$. 
Thus inductive application of Lemma \ref{LemmaInductivePolyhedra} gives that the $n$-simplex admits a polyhedral decomposition into $\left(\frac{Ln}{D}\right)^n$ polyhedra each bounded by at most $2n$ hyperplanes. By Lemma \ref{LemmaEuclideanSubdivide} this means each polyhedron can be triangulated by at most $2^{2n^2+2n}$ hyperbolic $n$-simplices.
Hence the total number of hyperbolic $n$-simplices required to triangulate a hyperbolic $n$-simplex of edge length less than $L>0$ into $n$-simplices of diameter less than $D$ is bounded above by
\[ 2^{2n^2+2n}\left(\frac{Ln}{D}\right)^n.\]

Thus, for $F$ as in Theorem \ref{InjRadSummary} we can subdivide a lift of $F(\Delta)$ as in Lemma \ref{LemmaInductivePolyhedra} for each $\Delta$ in the triangulation of $M$ and the preimage of this subdivision is a subdivision of $\Delta$ with the required property. Thus substituting in our known bounds on the edge lengths of the lift of $F(\Delta)$, summing over all $t$ $n$-simplices and simplifying big O notation gives the desired bound:

\[ t2^{2n^2+2n}\left(\frac{(nt)^{O(n^4t)}n}{D}\right)^n \leq (nt)^{O(n^4t)} \frac{1}{D^n}.  \]

\end{proof}

This theorem then combines with our bound on injectivity radius to give the following corollary.

\begin{corollary} \label{CorollaryGeometricEmbedded}
Given $(M,\mathcal{T})$ a triangulated $n$-manifold homeomorphic to a hyperbolic n-manifold $X$, $F : M \rightarrow X$ the homotopy equivalence defined in Theorem \ref{InjRadSummary}, we can find a subdivision of each simplex of $(M,\mathcal{T})$ such that $F$ restricted to each simplex has image a geodesically embedded hyperbolic tetrahedron of diameter less than half the injectivity radius of the manifold or even $1/c$ times the injectivity radius for some integer $c\geq 1$.

Furthermore, in the n-dimensional case, the union of these subdivisions consists of less than the following number of $n$-simplices:
\[ c^n2^{(nt)^{O(n^4t)}}.\] 

In the $3$-dimensional case, the bound improves to

\[ c^nt^{O(t)}\]
\end{corollary}

\begin{proof}
This is a simple corollary of Theorem \ref{TheoremInjRadBound} and Lemma \ref{GeometricImmersed} if we set $D = R/2c$ and simplify big O notation. In the 3-dimensional case, we instead apply Corollary \ref{3MfdInjRad}.
That the $n$-simplices are embedded follows from Remark \ref{RemarkInjRadEmbed}.

\end{proof}

So we now have a collection of embedded hyperbolic tetrahedra which cover our hyperbolic manifold $X$ and can intersect pairwise in at most one component and of which there are at most $T$ with $T \leq  2^{nt^{O(n^4t)}}$ if $n\geq 4$ or $T \leq t^{O(t)}$ if $n=3$. 

Our next task is to triangulate the intersections of these tetrahedra to give a true geometric triangulation of $X$, the preimage of this triangulation is then a triangulation of $M$ which has the geometric triangulation of $X$ as its simplicial quotient.

\begin{lemma}\label{GeometricSimplicialQuotient}
Suppose we are given $T$ geodesic $n$-simplices  cover a closed hyperbolic $n$-manifold $X$ such that the $n$-simplices  have diameter bounded by half the injectivity radius of the manifold. 
Then there is a geometric triangulation of $X$ such that the interior of each geometric $n$-simplex of the triangulation is either contained within or disjoint from the interior of each simplex of the cover.
This triangulation has size bounded by 
\[ T 2^{Tan^2} \]
 for some constant $a$.

\end{lemma}

\begin{proof}

We build up this triangulation starting with a single geometric $n$-simplex and  by adding one of our $T$ geometric $n$-simplices (which we shall call the ``old" $n$-simplices) at a time and subdividing so that the union so far is geometrically triangulated. Consider the first two $n$-simplices, they intersect in one component due to the requirement on injectivity radius. Now consider this pair in the universal cover of hyperbolic space and extend each of their faces to the hyperbolic hyperplanes which contain them, then we have $2n +2$ such hyperplanes, and choosing half spaces and taking their intersection yields at most $2^{2n+2}$ convex polyhedra, a subset of which tiles the union of two $n$-simplices. In the Klein ball model for the universal cover these polyhedra are indeed Euclidean and so we can geometrically triangulate each polyhedron by at most 
\[2^{7n^2} \]
$n$-simplices  by Lemma \ref{LemmaEuclideanSubdivide} and thus we get a geometric triangulation of the union of the two $n$-simplices by 
\[ 2^{2n+2} 2^{7n^2} < 2^{a(n^2)}  \]
for some constant $a$. If we then add a third tetrahedron, we need to intersect it with all these $2^{a(n^2)}$ tetrahedra and each intersection yields at most $2^{a(n^2)}$ tetrahedra itself, making $\leq \left(2^{a(n^2)}\right)^2$ $n$-simplices so far. In the end we are left with a geometric triangulation of $X$ by at most
\[\left( 2^{a(n^2)}\right)^T = 2^ {Tan^2} \] 
tetrahedra.

Because new tetrahedra either lie entirely in old tetrahedra or their complements this triangulation is such that the preimage of it under $F$ is a subdivision of the original triangulation of $M$, and the map $F$ is simplicial when considered as a map from this subdivision to the newly subdivided $X$. Each of the $ 2^ {Tan^2}$ new $n$-simplices can lie in at most $T$ old $n$-simplices and so has at most $T$ preimages in $M$, hence the subdivision of $M$ is by at most $T2^ {Tan^2}$ $n$-simplices. 

\end{proof}

\begin{figure}[ht]
\centering
    \includegraphics[width = 0.8\linewidth]{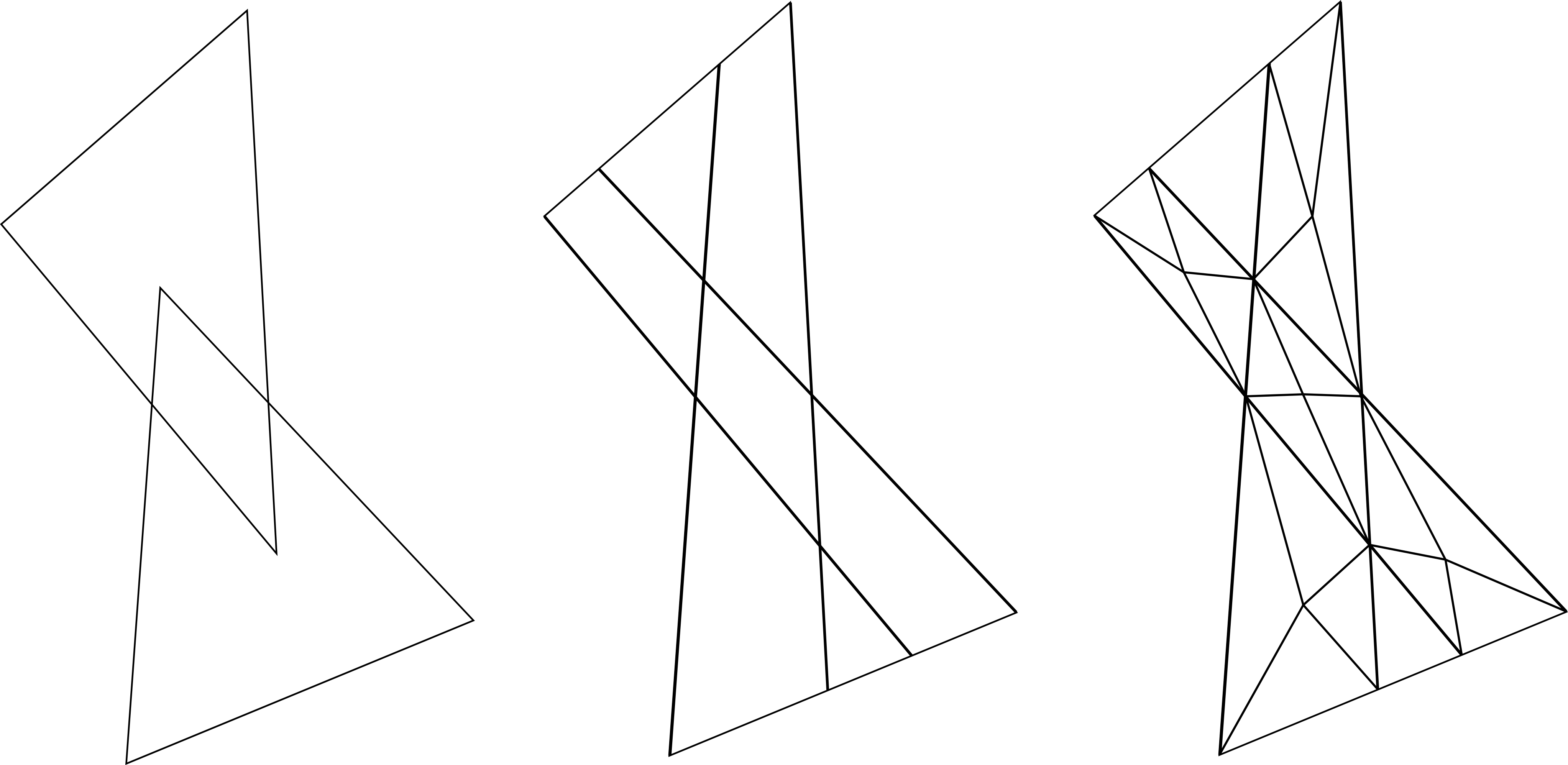}
    \caption{The intersection of two simplices can create non-convex (and hence hard to triangulate) polyhedra. Intersecting half spaces instead unfortunately creates more polyhedra, but  ensures they are all convex and thus easier to triangulate.}
\end{figure}

\begin{remark}\label{RemarkInjRadTriangulation}
Note that by Corollary \ref{CorollaryGeometricEmbedded} the edge lengths of the tetrahedra in this final geometric triangulation are less than the injectivity radius of the manifold.
\end{remark}

The following summarises the results of the section.
\begin{corollary} \label{GeometricSubdivisionBound}
Suppose we are given a manifold $M$ triangulated by $t$ tetrahedra such that $M$ admits a hyperbolic structure. Then $M$ admits a geometric triangulation where the number of tetrahedra in the geometric triangulation is bounded by 
\[2^{2^{(nt)^{O(n^4t)}}} \]
and furthermore this triangulation is a simplicial quotient of some subdivision of the triangulation of $M$ where the subdivision has its number of tetrahedra bounded by the same bound.

In dimension 3, the bound can be improved to 

\[ 2^{t^{O(t)}} \]
\end{corollary}

\begin{proof}
This follows from using the value of $T$ from Corollary  \ref{CorollaryGeometricEmbedded} (setting $c=2$) in Lemma \ref{GeometricSimplicialQuotient} and simplifying big O notation.
\end{proof}

We now have bounds on both the combinatorial complexity (number of tetrahedra in the subdivision of which our manifold is a simplicial quotient) and the geometric complexity (edge lengths are bounded above by half the injectivity radius). We shall use the first bound in Section \ref{SecFindingGeometricTriangulations} to build a list on which our manifold must appear. In Section \ref{SecComparingGeometricTriangulations} the bound on geometric complexity will allow us to navigate from one geometric triangulation to another in bounded time.

\section{Finding the Geometric Triangulation}\label{SecFindingGeometricTriangulations}

We now return to the 3-dimensional case.
A subtle point that we have thus far obscured is that we have used that Grigoriev's Theorem  (\autoref{Grigoriev}) gives a ``bounded" solution for every component of the solution set and that one of the components contains a faithful lattice representation to prove the existence of a ``bounded" faithful discrete representation. However, we gave (and know of) no general method for finding which of the solutions provided by Grigoriev's algorithm was our desired faithful discrete representation.

Thus we need another route to get to the geometric triangulation, and we will use the bounds on its complexity that we found in earlier sections. First for ease of reading we shall recap the major results of our previous sections that culminate in our algorithm. These are summarised in Theorem \ref{TheoremRecap}.

In Section \ref{SecGeometricQuotient} we saw that given a manifold $M$ triangulated by $t$ tetrahedra, there exists a geometric triangulation of $M$ which is a simplicial quotient of a subdivision of $M$ consisting of at most
\[ 2^{t^{O(t)}} \]
tetrahedra. This is exactly Corollary \ref{GeometricSubdivisionBound}.

The results of Section \ref{SecPachner}, in particular Theorem \ref{SubdivisionMoves} then tell us that we can perform a sequence of composite moves consisting of at most
\[  2^{t^{O(t)}}  \] 
Pachner moves to take us from the original triangulation of $M$ to the  subdivision with geometric quotient described above. All sequences of the form described in Section \ref{SecPachner} end in subdivisions of $M$. So if we perform all such sequences of composite moves which have their length in Pachner moves bounded in this way we get a list of subdivisions of $M$, at least one of which admits a geometric triangulation as a simplicial quotient, we can even discard any subdivisions whose size is bigger than our subdivision size bound.
As the number of possible moves at each step is on the order of at most 
\[ 2^{t^{O(t)}} \]
the total length of the list as well as the runtime of an algorithm producing it is bounded by
\[ 2^{2^{t^{O(t)}}} .\]

Given one subdivision $N$ of $M$ from the list described above, we showed in Section \ref{SecGeometricTriangulations}, in particular Corollary \ref{ListOfGeometricTriangulations} and Lemma \ref{DegreeOneList}, that we can list all oriented  gluings which can be found as a simplicial quotient (where the quotient map is degree one) of $N$ and which admit a structure of a geometric triangulation. In fact, for each such item of the list, the information of a geometric triangulation includes the vertices of a triangulation of a fundamental polyhedron in $\h^3$ and matrices in $PSL(2,\C)$ which perform the face pairings. This list is found in time exponential in the number of tetrahedra in the subdivision : 
\[ 2^{2^{t^{O(t)}}}.\]
 These gluings may or not be homeomorphic to the original manifold $M$ but they do admit degree one maps from $N$ and thus from $M$.
Now performing this for each subdivision on the list has run time bounded by

\[ \left({2^{2^{t^{O(t)}}}}\right) ^2 \]
but this once again simplifies to

\[ {2^{2^{t^{O(t)}}}}. \]
The above is thus a bound on the number of possible candidates in our list as well as a bound on the runtime of the algorithm creating the list.

We summarise this in the following Theorem.

\begin{theorem} \label{TheoremRecap}
Given an oriented closed hyperbolic $3$-manifold triangulated by $t$ tetrahedra, we can construct a list of geometrically triangulated oriented closed hyperbolic $3$-manifolds $X_i$ paired with degree one simplicial maps from some subdivision of $M$ to $X_i$.

This list is constructed in time bounded by
\[{2^{2^{t^{O(t)}}}} \]
and this is also a bound on the length of the list. Each triangulation has its number of tetrahedra bounded by

\[ {2^{t^{O(t)}}}.\]

Furthermore, the list contains a geometric triangulation which is simplicially isomorphic to the geometric triangulation defined in Section \ref{SecGeometricQuotient}, in particular Corollary \ref{GeometricSubdivisionBound}.
\end{theorem}

\begin{defn}[Candidate Manifolds]
We shall call the manifolds in the list created in the above theorem candidate manifolds, as they are so far the best candidates to be a true geometric triangulation of our original manifold $M$.
\end{defn}

As we know that all these manifolds are degree one simplicial quotients of subdivisions of $M$, all their fundamental groups admit surjections from $\pi_1 (M)$. 
Note that as hyperbolic $3$-manifold groups are residually finite, they are Hopfian, that is they don't admit self-surjections with non-trivial kernel (see \cite{Malcev40} or chapter 4 of \cite{LyndonSchupp77}). Thus, if some $Y$ in our list also admits $\pi_1$ surjections to the rest of the list and some $X$ in our list is homeomorphic to $M$, then there exists a chain of surjections:

\[ \pi_1(M) \twoheadrightarrow \pi_1(Y) \twoheadrightarrow \pi_1(X) \cong \pi_1(M). \]

The induced surjection $\pi_1(M) \twoheadrightarrow \pi_1(M)$ is an isomorphism and hence $\pi_1(M) \cong \pi_1(Y)$. Thus, by Mostow Rigidity, $M$ and $Y$ are homeomorphic.

Our plan moving forward is to identify such surjections by understanding the images of a given simplicial generating set of $\pi_1(M)$. For each candidate manifold $X$ in our  list the quotient map $M \rightarrow X$ is simplicial, so a simplicial generating set of $M$ maps to a simplicial generating set of $X$. However the algorithm which provides us with the hyperbolic structures on our list entries doesn't provide us with a simplicial generating set, it provides us with a set of generators for a lattice in $Isom^+(\h^3)$ or equivalently a generating set for the deck group of $M$.

 We now show how to relate these two notions, for this we slip back into the $n$-dimensional setting.

First, we recall some notions from Section \ref{SecPolyhedron}.
In that section we worked with a model for the geometry of the tetrahedra in a 3-dimensional triangulation but here we shall work purely with the combinatorial data of the $n$-dimensional triangulation of a hyperbolic manifold $X$.

\begin{defn} [Face Pairings] \label{FacePairings}
Recall that given some spanning tree $\Lambda'$ of the dual graph $\Lambda(X)$ to a triangulation of a hyperbolic $n$-manifold $X$, we can define a polyhedron $Y$ formed by ungluing the $n$-simplices which triangulate $X$ and then regluing only along the (codimension one) faces which correspond to edges of $\Lambda'$.

There is then a natural quotient map $\pi_Y: Y \rightarrow X$ which reglues along all faces to recover $X$.
As $Y$ is simply connected, given some choice of basepoint, $\pi_Y$ lifts to a map $\tilde{\pi}_Y : Y \rightarrow \tilde{X}$ with image a fundamental domain containing this basepoint for the deck group action on $\tilde{X}$.

As $X$ is a manifold, any translate of $\tilde{\pi_Y}(Y)$ can be reached by a path through codimension one faces and thus the deck group is generated by deck transformations which identify the faces of $\tilde{\pi_Y}(Y)$, we call these generators the face pairing generators or simply the face pairings of $Y$.

Note that these face pairings correspond to edges of the dual graph $\Lambda(X)$ which are not in $\Lambda'$. Indeed, if $\gamma$ is a loop in $\Lambda(X)$ which consist of a path in $\Lambda'$  and an edge $e \notin \Lambda'$ such that $e$ connects the faces $F$ and $F'$, then $\gamma$ lifts to a path which connects $\tilde{\pi_Y}(Y)$ to its translate under the deck transformation which identifies $F$ with $F'$. 
\end{defn}

\begin{figure}[ht]

    \begin{subfigure}[b]{.4\textwidth} 
    \centering
    \includegraphics[width = \linewidth]{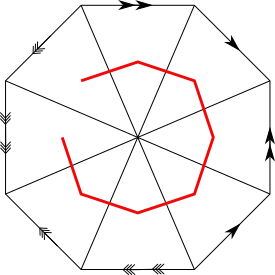}
    \caption{A certain choice of spanning tree $\Lambda'$ of the dual graph $\Lambda(X)$ where $X$ is the genus 2 surface.}
    \end{subfigure} 
    \hspace{0.04\textwidth}
    \begin{subfigure}[b]{.4\textwidth} 
    \centering
    \includegraphics[width = \linewidth]{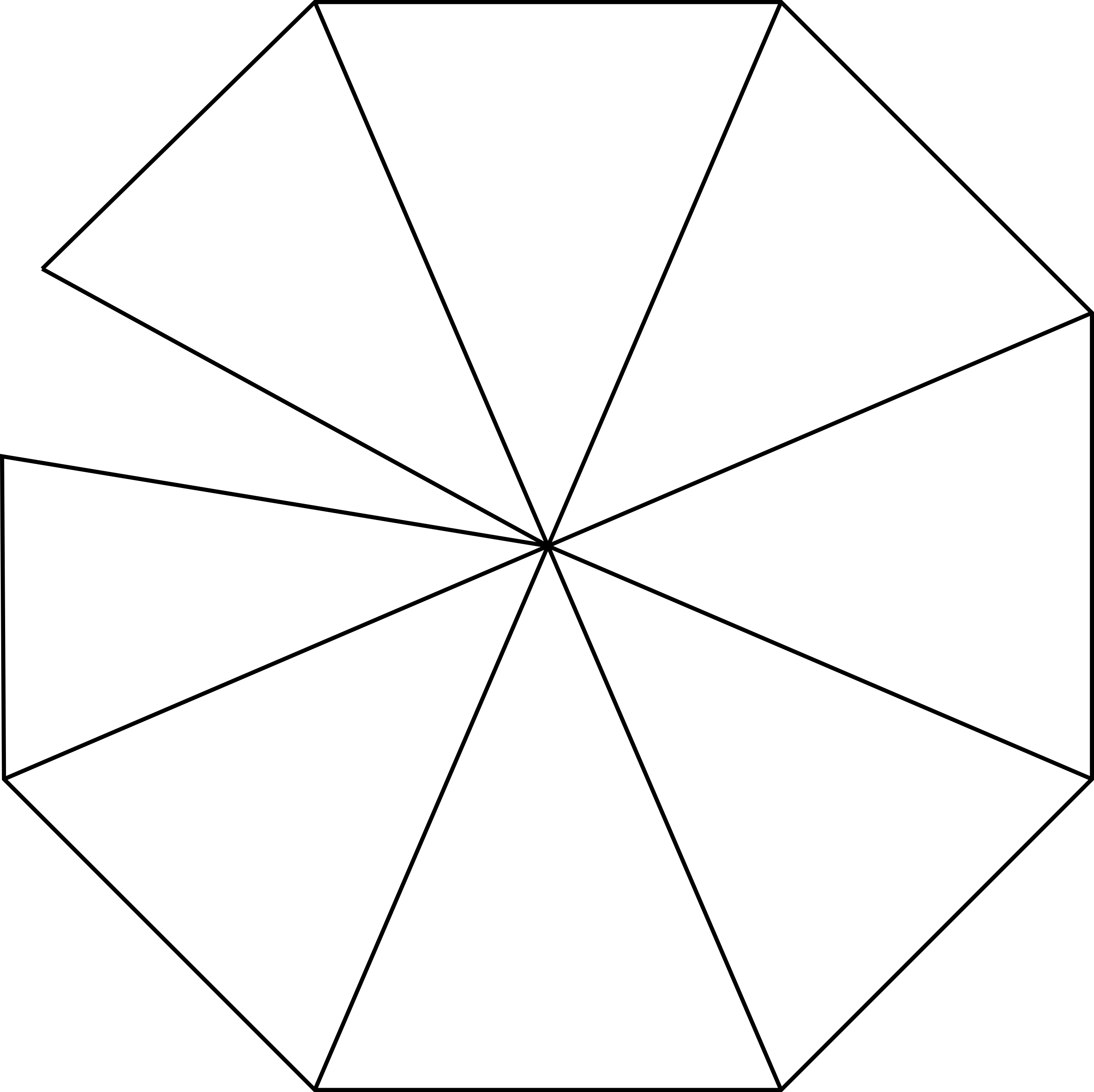}
    \caption{The polyhedron $Y$ associated to the spanning tree $\Lambda'$ defined in the first figure. }
    \end{subfigure}
   
    \caption{}
\end{figure} 

We now seek to relate the face pairings to the simplicial loops in $X$. We do this by taking a subdivision of $X$ which contains both the original triangulation and its dual graph.

\begin{defn} [Coning to the Boundary] 
Given a $k$-simplex $\Delta$, a subdivision of its boundary and a point $p$ in its interior, we can naturally associate $\Delta$ with 
\[C(\partial \Delta) = \partial \Delta \times [0,1] / (x,1) \sim (y,1) \forall x,y \]
 by identifying each line segment between $p$ and some boundary point $x$ with the interval $x \times [0,1]$.  
$C (\partial \Delta)$ has a natural structure as a triangulation, and we call this the triangulation formed by coning to the boundary of $\Delta$.
\end{defn}

\begin{defn}[$X'$ - a Partial Barycenteric Subdivision ]\label{Partial Barycentric Subdivision}
Define the subdivision $X'$ of $X$ as the triangulation formed by coning all $(n-1)$-simplices to their boundaries and then coning all $n$-simplices to their now subdivided boundaries. This can be thought of as a partial barycentric subdivision. Note that the number of $n$-simplices in this subdivision is on the order of $n^2T$, where $T$ is the number of $n$-simplices in the triangulation of $X$.

\end{defn}

Note that $\Lambda(X)$ doesn't actually embed in $X'$, but its barycentric subdivision naturally embeds as the graph spanned by the coning vertices. From this point on, we shall abuse notation to ignore this and simply say that a loop or subgraph in $\Lambda(X)$ embeds if its barycentric subdivision does. 

\begin{defn} [$\Gamma$ - a Spanning Tree for the 1-skeleton of $X'$]\label{SimplicialGeneratingSet}
 A spanning tree $\Lambda'$ of $\Lambda(X)$ also embeds in $X'$ and we can extend this to a spanning tree $\Gamma$ of the entire $1$-skeleton of $X'$  by adding only edges which contain the coning vertices of the $n$-simplices. We can perform this subdivision also on the polyhedron $Y$ to get a polyhedron $Y'$ and the quotient map $\pi_Y$ defined above induces a simplicial quotient map $\pi_{Y'}: Y' \rightarrow X'$. Note that by construction, the preimages $\Lambda'_{Y'} \subseteq \Gamma_{Y'}$ of $\Lambda' \subseteq \Gamma$ in $Y'$ are both still trees. 
 \end{defn}
 
 \begin{figure}[ht]

    \begin{subfigure}[b]{.4\textwidth} 
    \centering
    \includegraphics[width = \linewidth]{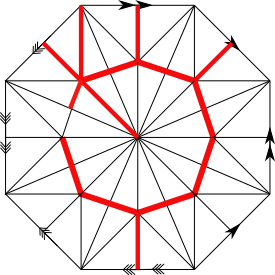}
    \subcaption{The spanning tree $\Gamma$ extends $\Lambda'$, note that it embeds also in $Y$ without being disconnected.}
    \end{subfigure} 
    \hspace{0.04\textwidth}
    \begin{subfigure}[b]{.4\textwidth} 
    \centering
    \includegraphics[width = \linewidth]{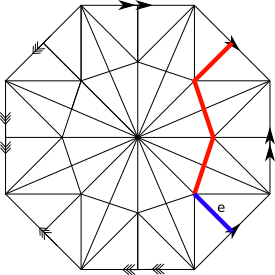}
    \subcaption{The face pairing loop associated to the ``single arrowhead" face pairing.}
    \end{subfigure}
   
    \caption{}
\end{figure}

\begin{remark} \label{simplicialgenerators}
Note that the face pairing loop $\gamma$ corresponding to a given face pairing of $Y$ defined in Definition \ref{FacePairings} embeds in $X'$ as the concatenation of an embedded path in $\Gamma$ and one edge not in $\Gamma$. The set of all simplicial loops meeting this description (an embedded path in $\Gamma$ and an edge $e$ not in $\Gamma$ connecting its endpoints) shall be called the simplicial generating set, and such a curve is called the simplicial generator associated to $e$. By construction the face pairing loops corresponding to the face pairings of $Y$ form a subset of the simplicial generating set of $X'$.

\end{remark}

\begin{lemma} \label{LemmaFacePairing}
These face pairing loops generate the fundamental group of $X$ and in fact the other generators can be expressed as at most $T$-fold products of these face pairings where $T$ is the number of $n$-simplices in the given triangulation of $X$. In fact we can find these products algorithmically in time bounded by
\[  (nT)^{O(nT)}. \]
\end{lemma} 

\begin{proof}

Let $e$ be some edge in $Y'$ not in $\Gamma_{Y'}$, but with at least one endpoint $w$ in $\Gamma_{Y'}$, note that all edges of $X'$ not in $\Gamma$ lift to such an edge. Let $v$ denote the other endpoint of $e$ and let $u$ be the unique vertex of $\Gamma_{Y'}$ such that $\pi_{Y'} (u) = \pi_{Y'}(v)$ in $X'$. The path $\gamma$ formed by concatenating a path in $\Gamma_{Y'}$ from $u$ to $w$ with $e$ projects to the loop $\pi_{Y'}(\gamma)$ in $X'$ which is exactly the simplicial generator associated to the edge $\pi_{Y'}(e)$.

If $\tilde{X}'$ denotes the universal cover of $X'$, then recall that $\tilde{\pi}_{Y'} : Y' \rightarrow \tilde{X}'$   is the lift of the quotient map $\pi_{Y'}:Y' \rightarrow X'$  and it maps $Y'$ to a fundamental polyhedron for the deck group action on $\tilde{X'}$. Consider $\tilde{u}, \tilde{v} = \tilde{\pi}_{Y'}(u), \tilde{\pi}_{Y'}(v)$ respectively, then the deck transformation mapping $\tilde{u}$ to $\tilde{v}$ is exactly the deck transformation corresponding to the loop $\pi_{Y'}(\gamma)$. Thus understanding how $\pi_{Y'}(\gamma)$ can be expressed as a product of face pairing loops is equivalent to understanding what sequence of face pairings maps $\tilde{u} $ to $\tilde{v}$.

Now as $X$ is a manifold, there is a sequence of distinct translates $\tilde{\pi}_{Y} (Y) = Y_1, \ldots , Y_k$   of $\tilde{\pi}_{Y}(Y)$, all of which contain $\tilde{v}$, each of which intersects the next along a codimension one face such that the deck transformation taking $Y_1$ to $Y_k$ maps $\tilde{u}$ to $\tilde{v}$. Define $F_i$ to be the face of $Y_1$ such that its translate in $Y_i$ connects $Y_i$ to $Y_{i+1}$. Denote by $E_i$ the face of $Y_1$ to which each $F_i$ is paired and let $g_{F_i}$ denote the face pairing deck transformation which maps $E_i$ to $F_i$. The deck transformation taking $Y_1$ to $Y_k$ is then exactly

\[ g_{F_1} \cdot \ldots \cdot g_{F_k} \]

Note that as there are $\leq (n+1)T$ translates of $Y_1$ around any one vertex, $k \leq (n+1)T$.

 Now as each translate contains $\tilde{v}$ 
 \[ \tilde{v} \in g_{F_1} \cdot \ldots \cdot g_{F_{k-i}}(Y_1)= Y_{k-i} \]  for all $i$, and so we can see that 
\[g_{F_{k-i}}^{-1}\cdot \ldots \cdot g_{F_1}^{-1} 
(\tilde{v} ) =: w_i \]   is a vertex of $Y_1$.
Furthermore as
\[g_{F_1} \cdot \ldots \cdot g_{F_k}( \tilde{u}) = \tilde{v} \]
we can see that
\[w_i = g_{F_{k-i}}^{-1}\cdot \ldots \cdot g_{F_1}^{-1} 
(\tilde{v} )= g_{F_{k-i+1}}
 \cdot \ldots \cdot g_{F_k} (\tilde{u} )  .\]
 
 Thus if we set $w_0 = \tilde{u}$, then each $w_i$ is a vertex of $Y_1$ related to the next in the sequence by a face pairing deck transformation. Thus to find a sequence of deck transformations taking $\tilde{u}$ to $\tilde{v}$ we can simply consider all sequences of $\leq (n+1)T$ face pairings, and build up a sequence of vertices of $Y_1$
 
 \[  w_i  = g_{F_{k-i+1}}
 \cdot \ldots \cdot g_{F_k} (\tilde{u} )\]
 for each and search for such sequences which take $\tilde{u}$ to $\tilde{v}$. Note that as $X$ is in fact a hyperbolic manifold, there is no deck transformation which fixes a point, and thus the fact that a deck transformation takes $\tilde{u}$ to $\tilde{v}$ uniquely determines it, so any such sequence of face pairings must give the chosen deck transformation.
   
 As both the total length of the sequence of pairings and the number of choices of pairing at each stage is bounded above by $(n+1)T$, and we know such a sequence always exists, we can find one in time $(nT)^{O(nT)}$. We do so for each of the edges not in $\Lambda_{P'}$ and thus we do this on the order of $n^2T$ times, and the result follows by simplifying big O notation.
 \end{proof}
 
\begin{figure}[ht]

    \centering
    \includegraphics[width = 0.7\linewidth]{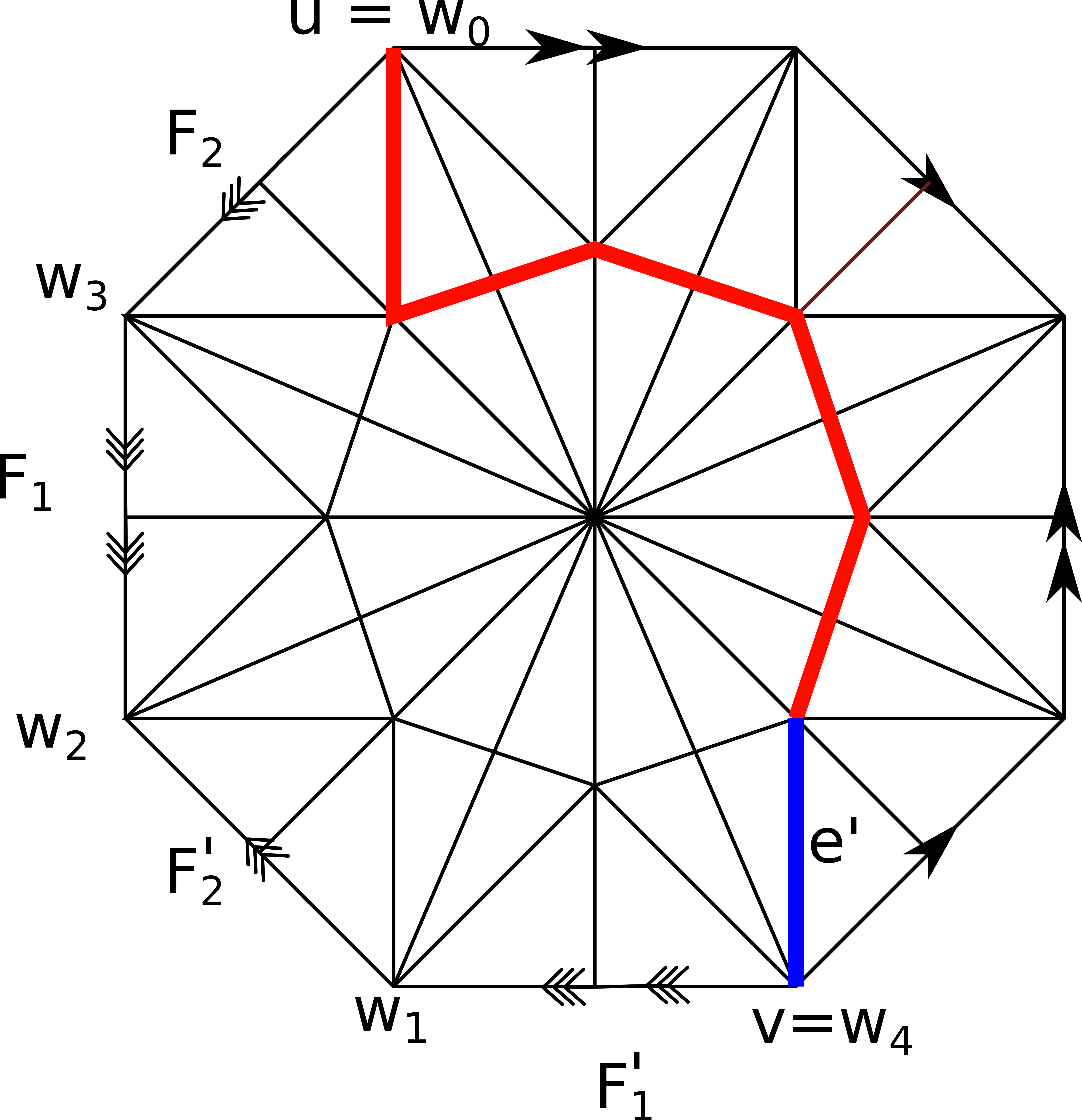}
    \caption{The simplicial loop corresponding to the edge $e'$ is not a face pairing loop and its corresponding deck transformation can be written as the product of face pairings $g_{F_1'}g_{F_2}g_{F_1}g_{F_2'}$. (Here $g_{F_i'}$ maps $F_i$ onto $F_i'$ and $g_{F_i} = g_{F_i'}^{-1}$).}
    
\end{figure}

 \begin{remark}
We remark here some technicalities concerning the above proof. If we only know the combinatorial data, then we are actually working with the face pairings of $Y$ rather than of its image $Y_1$ under the map $\tilde{\pi}_{Y}$. Thus we only actually know where $g_{F_{k-i}}$ sends the vertices of $E_i$ and no further information, so if at any point $g_{F_{k-i}}$ is not defined on $w_i$, we simply discard that sequence. A further technical quirk is that as we only know the combinatorial data coming from $Y$, there are some face pairings of $Y$ which are not face pairings of its image $Y_1$, but will still occur in our sequences, this occurs when the map $\tilde{\pi}_{Y}$ identifies a pair of faces $E, F$, in this case the corresponding deck transformation taking $E$ to $F$ would simply be the identity. 
 
These technicalities do not preclude the following two facts. Firstly, the existence of a sequence of face pairing deck transformations as defined above implies the existence of a sequence (still of length bounded above by $(n+1)T$) of face pairings of $Y$ which eventually relates $u$ to $v$. Furthermore, given a sequence of face pairings of $Y$ relating $u$ to $v$,  the corresponding sequence of face pairing deck transformations  (some of which may be the identity) translates $\tilde{u}$ to $\tilde{v}$.

\end{remark}

We now return to the 3-dimensional case. Looking for surjections between groups is not in general a simple task. It is certainly not a good idea to check through all possible maps from one group to the other and check whether they are surjective homomorphisms. We circumnavigate this by picking a specific `homomorphism' in advance to check. 
Let $X_1$, $X_2$ be candidate manifolds. If $X_1$ is indeed homeomorphic to $M$ and $f_{X_1}: M \rightarrow X_1$ is the quotient map then ${f_{X_1}}_* : \pi_1(M) \rightarrow \pi_1(X_1)$, the induced map on fundamental groups, is an isomorphism and so admits an inverse and so ${f_{X_2}}_* \circ {f_{X_1}}_*^{-1} : \pi_1(X_1) \rightarrow \pi_1(X_2)$ is a surjective homomorphism. Thus to find an $X_i$ which is homeomorphic to $M$ we don't need to check over all possible maps, we only need to check whether the above map exists and is a surjection for all targets $\pi_1(X_j)$.
Existence is important here as when ${f_{X_1}}_*$ is not an isomorphism then the above need not even be a well defined map.

\begin{figure}[ht]

    \begin{subfigure}[b]{.4\textwidth} 
    
    \includegraphics[width = \linewidth]{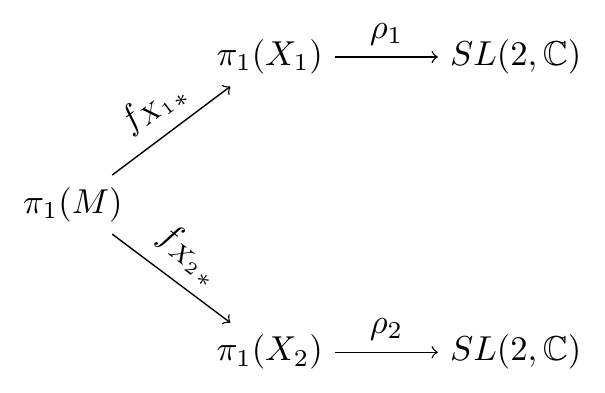}
    \caption{}
    \end{subfigure}
    \hspace{0.01\textwidth}
    \begin{subfigure}[b]{.4\textwidth} 
    \centering
    \includegraphics[width = \linewidth]{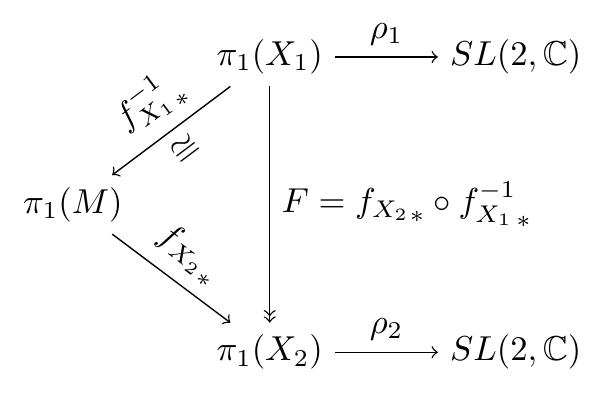}
    \caption{}
    \end{subfigure}
    
    \begin{subfigure}[b]{.5\textwidth} 
    \centering
    \includegraphics[width = \linewidth]{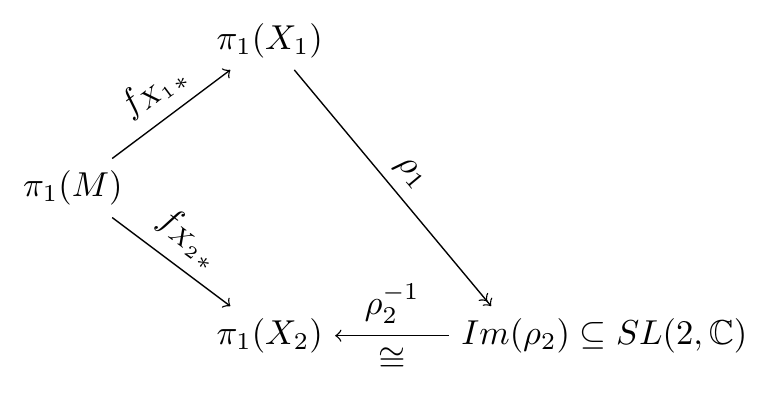}
    \caption{}
    \end{subfigure}

    \caption{Figure A shows the setup before we know anything about $X_1$,$X_2$ other than that they are each candidate manifolds. Figure B shows that if $X_1$ is homeomorphic to $M$ then $F$ is a well defined surjective homomorphism. Figure C shows how we get a surjective homomorphism whenever there is a solution to the combined system of polynomials.}
\end{figure} 

\begin{theorem} \label{TheoremFSearch}

Given a pair of candidate manifolds, $X_1, X_2$ triangulated by at most 
\[ T < 2^{2^{t^{O(t)}}}\]
 there is an algorithm which searches for a surjective homomorphism from $\pi_1(X_1)$ to $\pi_1(X_2)$. If  $X_1$  is homeomorphic to $M$, then it is guaranteed to find such a homomorphism, and the homomorphism is indeed the $F$ defined above. Otherwise, the algorithm may or may not find a surjective homomorphism, even when one does exist.

The runtime of this algorithm is bounded above by
\[T^{O(T^2)}. \] 
\end{theorem}

\begin{proof}
Let $S:=\{g_1, \ldots g_n \}$ be a simplicial generating set for $\pi_1(M)$ as in definition \ref{simplicialgenerators} for some spanning tree of the $1$-skeleton of $M$. Each $g_i$ then consists of at most $6t$ edges. As each $X_i$ is a quotient of a subdivision by at most $T$ simplices, the image of each $g_i$ in $\pi_1(X_i)$ can be expressed as a simplicial loop with at most $36tT$ edges, thus it is at most  a $36tT$-fold product of the simplicial generators of each $\pi_1(X_i)$. If $X_1$ is indeed homeomorphic to $M$ then $F$ is a well defined surjective homomorphism, and it takes the image ${f_{X_1}}_*(g_i)$ to ${f_{X_2}}_*(g_i)$.

We take as a starting point the very same system of polynomials that we used as input for the algorithm which found us the geometric structure on $X_2$. 

To this system we add polynomials with variables the real and imaginary parts of entries of complex matrices $H_1, \ldots H_k$ corresponding to the simplicial generators of $\pi_1(X_1)$. As $\pi_1(X_1)$ admits a presentation $ \pi_1(X_1) = \langle h_1, \ldots, h_n \mid r_1, \ldots r_n \rangle$, we also add polynomials corresponding to these relations, so for example, if $r_1 = h_3h_2h_1^{-1}$, then we'd add polynomials which encode the requirement that $H_3H_2H_1^{-1} =$ Id.

Then if we also require that these matrices satisfy that their determinants be $1$, we know that a solution to a system gives  both a geometric structure on $X_2$ and a representation of $\pi_1(X_1)$ into $SL(2,\C)$.

If for a faithful representation $\rho_2$ of $\pi_1(X_2)$ into $SL(2,\C)$, there is a representation $\rho_1$ of $\pi_1(X_1)$ where the image of each ${f_{X_1}}_*(g_j)$ is $\rho_2 ({f_{X_2}}_*(g_j))$, then this encodes a surjective homomorphism $\rho_2^{-1} \circ \rho_1 : \pi_1(X_1) \rightarrow \pi_1(X_2)$. 
Thus by combining the system of polynomials which we used to find a hyperbolic structure on $X_2$ (and which has variables corresponding to the lattice generators) with the system for representations of $\pi_1(X_1)$ as we did above, we can then add polynomials which relate the two systems.
As noted above, both $\rho_1 {f_{X_1}}_*(g_j)$ and $\rho_2 ({f_{X_2}}_*(g_j))$ are at most $36tT$-fold products of the images of the simplicial generators which are themselves at most $T$-fold products of the lattice generators (by Lemma, so the eight (two per complex variable of the complex matrices) polynomials required to encode the equation
 \[\rho_1 ({f_{X_1}}_*(g_j)) = \rho_2 ({f_{X_2}}_*(g_j)) \] are of degree $36tT^2$. Adding these polynomials makes a system that has a solution iff $X_2$ has a hyperbolic structure (which we already know it to have) such that $\pi_1(X_1)$ admits a representation which surjects onto the lattice which defines this hyperbolic structure of $X_2$.
 
Thus running the algorithm from Theorem \ref{Grigoriev} gives us a solution every time when $X_1 \cong M$ and any time it gives a solution, we know at least that $\pi_1(X_1)$ surjects onto $\pi_1(X_2)$. 
To find the runtime we just need to understand the complexity of the system of polynomials.
The original system for finding the hyperbolic structure on $X_2$ has complexity described as below, for some constant $C$:
\begin{itemize}

\item $\kappa \leq CT^2$
\item $N \leq CT$
\item $d \leq CT$
\item $M \leq CT$
 
\end{itemize}
The new system adds variables corresponding to the entries of matrices corresponding to each oriented $1$-simplex in the triangulation of $X_1$, hence the number of new variables is bounded by some constant multiple of $T$. 
We add polynomials for the relations, which come from the $2$-simplices of $X_2$ and thus have their number bounded by a constant multiple of $T$. The other polynomials come from the generators of $\pi_1(M)$ and thus have their number bounded by a constant multiple of $t << T$. The complexity of coefficients in the added polynomials is $\pm 1$ and the degree of the polynomials is bounded by $36tT$ which is on the order of $T$. Thus adding the new variables and polynomials has no impact on the complexity of the system up to altering the constant $C$. Thus an algorithm can find whether $F$ exists and is surjective in time bounded by
\[ T^{O(T^2)}\]
just as in Corollary \ref{HyperbolicStructureAlgorithm}.

\end{proof}

\begin{corollary} \label{AlgoHomeoCheck}
We can check whether a manifold $X_1$ in our list of candidate manifolds is homeomorphic to $M$ in time bounded by

\[ {2^{2^{t^{O(t)}}}}\]
\end{corollary}

\begin{proof}

As noted above, running the algorithm of Theorem \ref{TheoremFSearch} gives us a solution every time when $X_1 \cong M$, regardless of what $X_2$ is, and any time it gives us a solution, we know at least that $\pi_1(X_1)$ surjects onto $\pi_1(X_2)$. Thus, we know that running the algorithm above on each pair will return at least one candidate manifold which $\pi_1$-surjects onto all the others in time bounded by
\[T^{O(T)} .\]
Substituting the known bound on $T$ given by Corollary \ref{GeometricSubdivisionBound} gives that for each pair this algorithm has a runtime bound of the form
\[ \left({2^{t^{O(t)}}}\right)^{2^{t^{O(t)}}} \approx {2^{2^{t^{O(t)}}}}.\]

Now we saw above that the list of candidates has size bounded by the same bound. Thus we can check for all such surjections in time bounded by 
\[\left( {2^{2^{t^{O(t)}}}} \right) ^2 \approx {2^{2^{t^{O(t)}}}} .\]

We now repeat why it is that if all the above surjections exist this $X_1$ is homeomorphic to $M$. First note that on account of being hyperbolic manifolds, each element of the list is residually finite and hence Hopfian, this means that if two manifolds are such that the fundamental group of each admits a surjection onto the other, then these surjections are isomorphisms. We also know that the fundamental group of every manifold in the list admits a surjection from $\pi_1(M)$ and that one element of the list is homeomorphic to $M$ and thus that element, call it $Y_1$, has this same property. Thus $X_1$ and $Y_1$ are both such that their fundamental groups surject onto one another and are thus isomorphic. Now given that the manifolds are both hyperbolic, this isomorphism induces a homeomorphism by Mostow Rigidity. Thus $X_1$ is homeomorphic to $M$, as required. 

\end{proof}

We are now ready to bring all of this together into an algorithm which lists geometric triangulations of $M$. However, in Section \ref{SecGeometricQuotient}, we showed the existence of not only a geometric triangulation, but one with edge lengths bounded by injectivity radius, the following lemma will allow us to recover this property in our geometric triangulations also.

\begin{lemma} \label{LemmaEdgeLessInjRad}
Let $M$ be a 3-manifold triangulated by $T< 2^{t^{O(t)}}$ tetrahedra (and suppose $M$ also admits a triangulation by $t<T$ tetrahedra). We can check if $M$ admits the structure of a geometric triangulation of a hyperbolic $3$-manifold with edge lengths bounded above by half of the injectivity radius. If it does admit such a structure we can find it in time bounded by
\[ c2^{2^{t^{O(t)}}}.\] 
\end{lemma}

\begin{proof}

Note that we are simply reproving Corollary \ref{HyperbolicStructureAlgorithm}, with the extra requirement of this injectivity radius bound and with a better knowledge of how $T$ and $t$ interact.

Note that as $M$ admits a triangulation by $t$ tetrahedra, then if it admits a hyperbolic structure, its injectivity radius will be bounded below by
\[ \frac{1}{t^{O(t)}}\]
by Corollary \ref{3MfdInjRad}. Thus, all we need to do is add equations into the system of polynomials of Lemma \ref{LemmaSystemGeomTriang} which require that all the edge length variables be less than half the injectivity radius.  This introduces less than $6T$ new polynomials with coefficients of complexity bounded by $log_2(t^{O(t)})$. Thus the new system has size bounded as follows for some universal constant $C'$ (compare with the size of the system in Lemma \ref{LemmaSystemGeomTriang}).
\begin{itemize}

\item $\kappa \leq C'T^2$
\item $N \leq C'T$
\item $d \leq C'T$
\item $M \leq C'T$
 
\end{itemize}
and thus it follows by the same proof as Corollary \ref{HyperbolicStructureAlgorithm} and our known bounds on $T$ that the runtime of the algorithm is 
\[ C'T(C'T^2C'T)^{(C'T)^2}\leq  T^{O(T^2)} \leq  2^{2^{t^{O(t)}}}.  \]
\end{proof}

\begin{remark}
Note that here we could just as easily have bounded edge length by any multiple $1/c$ of injectivity radius, for some integer $c\geq 1$, as we did in Corollary \ref{CorollaryGeometricEmbedded}. The runtime here would be bounded by 

\[ c2^{2^{t^{O(t)}}}\]
as $c$ would only affect the complexity $M$ of the coefficients.
\end{remark}

We can now prove the following theorem

\begin{theorem}\label{TheoremListGeometricTriangulations}
Given a triangulation of a hyperbolic $3$-manifold $M$ by $t$ tetrahedra, then there is an algorithm which produces a geometric triangulations of $M$  consisting of less than
\[ 2^{t^{O(t)}} \] tetrahedra in time bounded by
\[ 2^{2^{t^{O(t)}}}.\]
Furthermore, this geometric triangulation is simplicially isomorphic to the triangulation defined in Lemma \ref{GeometricSimplicialQuotient}. Hence, this geometric triangulation has the property that all edge lengths are bounded above by the injectivity radius of the manifold.
\end{theorem}

\begin{proof}
An algorithm producing triangulation first simply produces the list of candidate manifolds as described at the start of this section. It then performs the algorithm of Corollary \ref{AlgoHomeoCheck} for each candidate manifold. The runtime of this algorithm and the length of the list of candidate manifolds is bounded by
\[ 2^{2^{t^{O(t)}}}\]
and thus this is a bound for the whole algorithm after simplifying big O notation.

The reason that we know one of our list must be simplicially isomorphic to the triangulation of Lemma \ref{GeometricSimplicialQuotient} is that we have listed all possible subdivisions and all possible simplicial quotients within the given complexity bound, and have only discarded elements of the list if they don't admit geometric structures or aren't homeomorphic to $M$.
Thus we know that for at least one element of the list, the algorithm  of Lemma \ref{LemmaEdgeLessInjRad} returns a positive result, a geometric triangulation with edge lengths bounded by the injectivity radius of the manifold.
\end{proof}

As an aside, it is interesting to note here that though the above theorem provides a geometric triangulation which is simplicially isomorphic to the one derived in Lemma \ref{GeometricSimplicialQuotient} and which shares the same property of edge lengths being less than injectivity radius, the geometric structures may still be different between the two.

\section{Comparing Geometric Triangulations}
\label{SecComparingGeometricTriangulations}

Given two closed hyperbolic $3$-manifolds $M, N$ we have so far managed to produce bounded complexity geometric triangulations of these manifolds. Thus all that remains is to understand how to compare these geometric triangulations.
To do this we shall provide a bound on the number of Pachner moves between two geometric triangulations of the same manifold, and then perform all sequences of Pachner moves of length less than that bound on the first triangulation and check if the result after any of the sequences is the second triangulation. If so, they're the same manifold and if not they must be different.

To do this we use the following which is adapted from a more general theorem of Kalelkar and Phanse which holds in higher dimensions and for spherical and Euclidean manifolds as well.

\begin{theorem}[Kalelkar Phanse '19]

Let $M$ be a closed hyperbolic $3$-manifold with geometric triangulations $\T_1$ and $\T_2$ by $t_1, t_2$ tetrahedra respectively. Let $L$ be an upper bound on edge length and inj$(M)$ be the injectivity radius of $M$.

$K_1$ and $K_2$ are related by  $\leq f(t_1,t_2,L,\text{inj}(M))$ Pachner moves which do not remove common vertices. The function $f$ is defined as follows

\[ f(t_1,t_2,L,\text{inj}(M)) = 32 (24^{4+3m})t_1t_2(t_1 + t_2)\]

where $m$ is defined as a natural number greater than

\[ (2 cosh^2(L) +1) ln(L/ \text{inj}(M)) \]

\end{theorem}

\begin{corollary} \label{CorollaryInjRadKalelkar}
Let $M$ be a closed hyperbolic $3$-manifold with geometric triangulations $\T_1$ and $\T_2$ by $<T$ tetrahedra each and such that the lengths of the edges of tetrahedra in each triangulation are strictly less than the injectivity radius of the manifold. Then the number of Pachner moves required to get from one to the other is bounded by 
\[ O(T^3) .\]
\end{corollary}

\begin{proof}
Note that when $L \leq inj(M)$, the lower bound on $m$ becomes negative, as $ln(x)$ is negative for all $x<1$. Thus $m$ can be taken to be $0$ and so the result follows.
\end{proof}

%\begin{remark}
%\label{RemarkGeometricReplaceSimplicial}
%Note that Pachner moves affect only the combinatorial information of a geometric triangulation, so we can replace each of $\T_1$ and $\T_2$ by some simplicially isomorphic geometric (or even non-geometric) triangulations and the Pachner move bounds still hold. All that matters is that the two triangulations admit a geometric structure with the appropriate bounds on edge length. This is important to the proof of the following Theorem.
%\end{remark}

\begin{reptheorem} {MainTheorem}
If $M_1$ and $M_2$ are triangulated closed hyperbolic 3-manifolds, each triangulated by less than $t$ tetrahedra, then we can decide in time bounded by 
\[  2^{2^{t^{O(t)}}}\]
whether or not they are homeomorphic. 
\end{reptheorem}

\begin{proof}
Theorem \ref{TheoremListGeometricTriangulations} provides for each manifold $M_1$, $M_2$ a geometric triangulation with edge lengths bounded by the injectivity radius, doing so in time bounded by

\[ 2^{2^{t^{O(t)}}} .\]

We also know that the number of tetrahedra in each of the triangulations is bounded above by
\[ 2^{t^{O(t)}}.\]
Thus by Corollary \ref{CorollaryInjRadKalelkar} the number of Pachner moves required to take $\T_1$ to $\T_2$ is bounded above by
\[\left( 2^{t^{O(t)}}\right)^3 \approx 2^{t^{O(t)}}.\]
As the number of moves at each stage is polynomial in the number of tetrahedra, the number of possible sequences of this length is bounded by
\[ 2^{2^{t^{O(t)}}} \]
thus it is also a bound for the runtime of the algorithm which checks the output of all the sequences to see if any of them is simplicially isomorphic to $\T_2$. 

There are two possibilities, either the algorithm returns a match, in which case the manifolds are homeomorphic or it doesn't. If there is no match then there is no sequence of Pachner moves of length less than the given bound relating the two geometric triangulations of $M_1$ and $M_2$ which contradicts Corollary \ref{CorollaryInjRadKalelkar} and thus the two manifolds must be distinct. 
\end{proof}

\newpage
\newcommand{\etalchar}[1]{$^{#1}$}

\end{document}